\documentclass[12pt]{article}
\usepackage{amsthm,amssymb,amsmath}
\usepackage{fullpage}
\usepackage[T1]{fontenc}
\usepackage{graphicx}
\usepackage{subfigure}
\usepackage{morefloats}

\usepackage[usenames,dvipsnames]{xcolor}
\usepackage[pdfusetitle,
bookmarks=true,bookmarksnumbered=true,bookmarksopen=true,bookmarksopenlevel=1,
breaklinks=false,pdfborder={0 0 
  0},backref=false,colorlinks=true,linkcolor=NavyBlue,citecolor=NavyBlue,urlcolor=NavyBlue]
{hyperref}
\pdfpageheight\paperheight
\pdfpagewidth\paperwidth

\usepackage[capitalise]{cleveref}
\crefformat{equation}{(#2#1#3)}

\numberwithin{equation}{section}

\newtheorem{theorem}{Theorem}[section]
\newtheorem{lemma}[theorem]{Lemma}

\newtheorem{corollary}[theorem]{Corollary}
\newtheorem{proposition}[theorem]{Proposition}

\newtheorem{definition}[theorem]{Definition}

\theoremstyle{definition}
\newtheorem{remark}[theorem]{Remark}
\newtheorem*{note*}{Note}

\newcommand{\R}{\mathbb{R}}
\renewcommand{\H}{\mathbb{H}}
\newcommand{\C}{\mathbb{C}}
\newcommand{\Z}{\mathbb{Z}}
\newcommand{\T}{\mathbb{T}}
\newcommand{\cB}{\mathcal{B}}

\newcommand{\cP}{\mathcal{P}}
\newcommand{\cE}{\mathcal{E}}
\newcommand{\cI}{\mathcal{I}}
\newcommand{\cS}{\mathcal{S}}
\renewcommand{\epsilon}{\varepsilon}

\renewcommand{\Im}{\operatorname{Im}}
\renewcommand{\Re}{\operatorname{Re}}

\DeclareMathOperator{\mes}{mes}
\DeclareMathOperator{\Vol}{Vol}

\DeclareMathOperator{\DC}{DC}
\DeclareMathOperator{\Car}{Car}
\DeclareMathOperator{\Proj}{Proj}

\newcommand{\norm}[1]{\lVert#1\rVert} 
\newcommand{\hsnorm}[1]{\lVert#1\rVert_{\mathrm{HS}}}

\begin{document}

\title{Non-Perturbative Localization with Quasiperiodic Potential in
  Continuous Time}

\author{Ilia~Binder\footnote{I.B. was supported by an NSERC Discovery grant},
  Damir~Kinzebulatov and Mircea~Voda}

\date{}

\maketitle

\begin{abstract}
  We consider continuous one-dimensional multifrequency Schr\"odinger
  operators, with analytic potential, and prove Anderson localization
  in the regime of positive Lyapunov exponent for almost all phases
  and almost all Diophantine frequencies.
\end{abstract}

\tableofcontents

\section{Introduction}
\label{sec:intro}
We consider the family of operators on $ L^2(\R) $ given by
\begin{equation}\label{eq:Schroedinger}
  [H(\theta,\omega)y](t)=-y''(t)+V(t,\theta+t \omega)y(t),
\end{equation}
where the potential $V:\T\times\T^d\to \R $ is analytic
($ \T:=\R/\Z $, $ d\ge 1 $), $ \theta \in \T^d $, and $ \omega $
satisfies a Diophantine condition. More precisely we will work with
frequency vectors in a set defined by
\begin{equation*}
  \DC := \{ \omega\in \T^d : \norm{k\cdot \omega}\ge c |k|^{-A}, k\in \Z^d\setminus
  \{ 0  \}  \}.
\end{equation*} 
for some $ A>d $. We used $ \norm{\cdot} $ to denote the distance to
the nearest integer and $ |\cdot| $ for the $ \sup $-norm on $ \Z^d $.
We will use $ L(\omega,E) $ to denote the Lyapunov exponent associated
with our operators (see \cref{sec:transfer-matrix-formalism} for the
definition).  Our main result is as follows.

\begin{theorem}\label{thm:localization}
  Assume that $ L(\omega,E)>0 $ for all
  $ (\omega,E)\in \DC\times [E',E''] $.  Then for almost all phases
  $ \theta\in \T^d $ and almost all frequency vectors
  $ \omega\in \DC $ the part of the spectrum of $ H(\theta,\omega) $
  contained in $ [E',E''] $ is pure point with exponentially decaying
  eigenfunctions.
\end{theorem}

Non-perturbative localization results (in the sense that one only requires positivity of the Lyapunov exponent)
are well known for discrete Schr\"odinger operators, dating back to work by Jitomirskaya
\cite{Jit99} for the Almost Mathieu operator and by Bourgain, and
Goldstein \cite{BouGol00} for general analytic potentials. For continuous Schr\"odinger operators the only known
result, due to
Fr\"ohlich, Spencer, and Wittwer \cite{FroSpeWit90}, deals with
potentials of the form
\begin{equation}\label{eq:cos-potential}
  K^2(\cos(2\pi t)+\cos(2\pi(\theta+t\omega)))
\end{equation}
with $ K $ sufficiently large. At the same time, there was no reason to expect that the discrete results don't carry
to the continuous case (indeed, \cite{FroSpeWit90} treats both the discrete and the continuous cases).
Our motivation for considering this problem stems from the recent work on the inverse spectral theory for continuous
quasiperiodic Schr\"odinger operators started by Damanik and Goldstein \cite{DamGol14}. Their work is in a
perturbative setting (assuming a small coupling constant) and it is natural to try to extend it to a
non-perturbative setting. On one hand, one can try to prove the results of  \cite{DamGol14} in the discrete case and
then make use of the non-perturbative theory available there (though, it is known that for the inverse spectral
problem one should consider Jacobi operators instead of Schr\"odinger operators). On the other hand, one is
motivated to develop the non-perturbative theory in the continuous setting. Our work is a step in this direction.

The fact that \cref{thm:localization} is non-vacuous, i.e. there exists a portion
of the spectrum where the theorem applies, follows from work on the
positivity of the Lyapunov exponent, by Sorets and Spencer
\cite{SorSpe91} for the case when the potential is of the form \cref{eq:cos-potential}
and by Bjerkl\"ov \cite{Bje06} for general analytic potentials that
assume their minimum value only finitely many times. The result of \cite{Bje06} is perturbative,
in the sense that the largeness of the coupling constant depends on the frequency vector. It would be interesting
to see whether it is possible to obtain a non-perturbative result on the positivity of the Lyapunov exponent
assuming only that $ V $ is non-constant. Such results are known in the discrete case (see \cite{SorSpe91} for
$ d=1 $, and \cite{Bou05a} for $ d\ge 1 $). We note that such a result in the continuous case may not be completely
non-perturbative, because it is well known that the Lyapunov exponent vanishes at high energies (see \cite{Eli92}),
and it seems that the transition between energies with positive Lyapunov
exponent and energies with zero Lyapunov exponent is of a perturbative nature (see \cite[Rem. 1.2]{YouZho14}).

The proof of \cref{thm:localization} follows along the same lines as in the
discrete case. The main ingredients are a result on the exponential
decay of the finite interval Green's function (see
\cref{prop:Green-decay}) and a result on elimination of resonances
(see \cref{prop:elimination-fixed-phase}). The basis for these results is
a large deviations estimate for the logarithm of the norm of the transfer matrix
(see \cref{thm:ldt}), which one obtains immediately from the discrete case. 
The fact that the large deviations estimate implies the decay of Green's function (on some interval) is
trivial in the discrete setting, but not in the continuous setting, as can be seen from the proof of
\cref{prop:Green-decay}. 
To eliminate the resonances we use the strategy of \cite{BouGol00} based on semialgebraic sets.
The main difference from the discrete case is that to obtain semialgebraic sets it is not
enough to approximate the potential by a polynomial. Instead we have to directly approximate
the entries of the transfer matrix. This leads to a different set-up for the elimination of
resonances (see \cref{rem:discrete-vs-continuous-elimination}) which, unlike the discrete case,
requires the use of Cartan's estimate as in \cite{GolSch08}. Furthermore, the approximating polynomials for
the entries of the transfer matrix cannot be obtained via Fourier series as in the discrete case. We use Faber
series instead.

The structure of the paper is as follows. The definition and all of
the needed properties for the transfer matrix and the Laypunov
exponent are presented in \cref{sec:transfer-matrix-formalism} and
\cref{sec:transfer-matrix-properties}. The exponential decay of
Green's function is established in \cref{sec:decay}. \cref{sec:Cartan}
and \cref{sec:semialgebraic} contain all the preliminary work needed
for the elimination of resonances result from
\cref{sec:elimination}. Finally the proof of \cref{thm:localization}
is obtained in \cref{sec:main-proof}.

\section{Transfer Matrix
  Formalism}\label{sec:transfer-matrix-formalism}

Consider the eigenvalue equation
\begin{equation}\label{eq:eigenvalue-equation}
  -y''(t)+V(t,\theta+t\omega) y(t) = E y(t),\quad  \theta\in \T^d.
\end{equation}
Let
$ u_a=u_a(\cdot;\theta,\omega,E),
v_a=v_a(\cdot;\theta,\omega,E) $
be the solutions of \cref{eq:eigenvalue-equation} satisfying
\begin{equation}\label{eq:initial-conditions}
  u_a(a)=1,u_a'(a)=0,\quad v_a(a)=0,v_a'(a)=1.
\end{equation}
Any solution $ y $ of \cref{eq:eigenvalue-equation} satisfies
\begin{equation*}
  \begin{bmatrix}
	y(b) \\ y'(b)
  \end{bmatrix}
  = M_{[a,b]} \begin{bmatrix}
	y(a) \\ y'(a)
  \end{bmatrix},\quad a\le b
\end{equation*}
where the transfer matrix is defined by
\begin{equation*}
  M_{[a,b]}=M_{[a,b]}(\theta,\omega,E):=
  \begin{bmatrix}
	u_a(b) & v_a(b)\\
    u_a'(b) & v_a'(b)
  \end{bmatrix}.
\end{equation*}
The transfer matrix satisfies
\begin{gather}
  M_{[a,b]}=M_{[t,b]}M_{[a,t]},\quad a\le t \le b\label{eq:M-semigroup}\\
  M_{n+[a,b]}(\theta,\omega,E)=M_{[a,b]}(\theta+n\omega,\omega,E),\quad
  n\in\Z\label{eq:M-shift}.
\end{gather}

\begin{remark}
  The fact that we can only use discrete shifts in \cref{eq:M-shift},
  stems from the fact that we are working with
  potentials of the form $ V(t,\theta+t\omega) $ instead of just
  $ V(\theta+t\omega) $. The reason  we consider the more general
  potentials $ V(t,\theta+t\omega) $ is to be able to apply the result
  of Bjerkl\"ov \cite{Bje06} that guarantees that the statement of our
  main theorem is not vacuous. 
\end{remark}

Equation \cref{eq:eigenvalue-equation} can be re-written as
\begin{equation*}
  Y'(t)=A(t)Y(t),\quad
  Y(t)=\begin{bmatrix}
	y(t) \\ y'(t)
  \end{bmatrix}, \quad A(t)=\begin{bmatrix}
	0 & 1 \\
    V(t,\theta+t\omega)-E & 0
  \end{bmatrix}.
\end{equation*}
As a consequence of Gr\"onwall's inequality one has
\begin{equation*}
  \norm{Y(b)}\le \exp \left( \int_a^b \norm{A(t)}\,dt \right)\norm{Y(a)}. 
\end{equation*}
Therefore we have
\begin{equation}\label{eq:transfer-matrix-bound}
  \log\norm{M_{[a,b]}}\le (b-a)C(V,|E|).
\end{equation}
Since
\begin{equation*}
  \det M_{[a,b]}=W(u_a,v_a)=1,
\end{equation*}
where $ W $ stands for the Wronskian, it follows that we also have
\begin{equation}\label{eq:inverse-transfer-matrix-bound}
  \log\norm{M_{[a,b]}^{-1}}\le (b-a)C(V,|E|).
\end{equation}
Using \cref{eq:M-semigroup}, \cref{eq:M-shift},
\cref{eq:transfer-matrix-bound}, and
\cref{eq:inverse-transfer-matrix-bound} we see that, as in the
discrete case, we have the following ``almost invariance'' property
\begin{equation}\label{eq:almost-invariance}
  \left|\log\norm{M_{[a,b]}(\theta,\omega,E)}-\log\norm{M_{[a,b]}(\theta+\omega,\omega,E)}\right|\le C(V,|E|).
\end{equation}
This property is essential for establishing a large deviations
estimate (see \cref{thm:ldt}).

The finite scale Lyapunov exponents are defined by
\begin{equation*}
  L_{[a,b]}(\omega,E)=\frac{1}{b-a}\int_{\T^d}\log\norm{M_{[a,b]}(\theta,\omega,E)}\,d\theta.
\end{equation*}
We let $ M_t=M_{[0,t]} $ and $ L_t=L_{[0,t]} $. By \cref{eq:M-semigroup} and \cref{eq:M-shift}, the sequence
$ (L_n)_{n\ge 1} $ is subadditive and so by Fekete's subadditive lemma
we can define the Lyapunov exponent
\begin{equation}\label{eq:Lyapunov-definition}
  L(\omega,E):=\lim_{n\to\infty} L_n(\omega,E)=\inf_{n\ge 1} L_n(\omega,E).
\end{equation}
We note that by Kingman's subadditive ergodic theorem we also have
\begin{equation}\label{eq:Lyapunov-Kingman}
  L(\omega,E) \stackrel{\text{a.s.}}{=}
  \lim_{n\to \infty}\frac{1}{n}\log\norm{M_n(\theta,\omega,E)},
\end{equation}
but we won't make use of this fact.

Let
\begin{equation}\label{eq:H_rho}
  \H_\rho=\{z\in \C: |\Im z|\le \rho\}.
\end{equation}
It is known that there exists $ \rho=\rho(V) $ such that $ V $ extends
to be complex analytic in a neighborhood of $ \H_\rho^{d+1} $ and
the extension remains periodic in the real direction. In particular,
this implies that
\begin{equation*}
  L_{[a,b]}(\eta,\omega,E)=\frac{1}{b-a}\int_{\T^d} \log\norm{M_{[a,b]}(\theta+i\eta,\omega,E)}\,d\theta
\end{equation*}
is well defined for all $ E\in \C $ and $ \omega\in\C^d $,
$ \eta\in \R^d $ such that
\begin{equation*}
  \max(|a|,|b|)\norm{\Im\omega}\le \rho/2,\quad  \norm{\eta}\le \rho/2 .
\end{equation*}
As before, we can define
\begin{equation*}
  L(\eta,\omega,E)=\lim_{n\to\infty} L_n(\eta,\omega,E)=\inf_{n\ge 1} L_n(\eta,\omega,E).
\end{equation*}

\section{Properties of the Transfer
  Matrix}\label{sec:transfer-matrix-properties}

All the results of this section are analogues of results obtained in
the discrete setting by Goldstein and Schlag \cite{GolSch01},
\cite{GolSch08} (a large deviations estimate and an uniform upper
bound can also be found in \cite{BouGol00}).  Since the results we
need from \cite{GolSch01} were already obtained in a multifrequency
setting with general Diophantine condition we only discuss their
proofs in \cref{sec:appendix} (which includes proofs for \cref{thm:ldt}, \cref{prop:Lyapunov-convergence-rate}, and
\cref{lem:Lyapunov-increasing}). The results we need from
\cite{GolSch08} were obtained in a single frequency setting with a
strong Diophantine condition, so we present their proofs in this
section. In either case, the proofs are only given in the interest of
clarity, as they follow along the same lines as in the originals.

Let
\begin{equation*}
  \DC_t := \{ \omega\in \T^d : \norm{k\cdot \omega}\ge c |k|^{-A}, k\in \Z^d\setminus
  \{ 0  \}, |k|\le t  \}.
\end{equation*}
For the purposes of the semialgebraic approximation it is important to
keep track of the fact that bmost finite scale results require
only a finite Diophantine condition, as above, and only the positivity
of the finite interval Lyapunov exponent (if needed at all).

In what follows we will always assume that the intervals we work on
are finite and non-trivial. Also note that all the results are only
effective when the size of the interval is large enough. For smaller
sizes the constants can be adjusted so that the results hold
trivially.

The main ingredient for both the decay of Green's function and the
elimination of resonances is the following large deviations estimate.
\begin{theorem}\label{thm:ldt}
  Let $ I=[a,b] $ and $ \epsilon>0 $. Then for any $ \omega\in \DC_{|I|} $, $ E\in \C $,
  and $ \eta\in \R^d $, $ \norm{\eta}\le \rho(V) $, we have
  \begin{equation*}
    \mes \{ \theta\in \T^d: |\log\norm{M_I (\theta+i\eta,\omega,E)}-|I|L_I(\eta,\omega,E)|
    \ge \epsilon|I|^{1-\sigma}  \}
    \le C\exp(-c|I|^{\sigma}),
  \end{equation*}
  with $ c=c(V,d,\DC,|E|,\epsilon) $, $ C=C(V,d,\DC,|E|) $, and
  $ \sigma=\sigma(d,\DC)\in(0,1) $.
\end{theorem}

\begin{note*}
  From now on $ \sigma $ will denote the constant from \cref{thm:ldt}. 
\end{note*}

We will need the following result to relate the Lyapunov exponent
with the finite interval Lyapunov exponents, and the finite interval
Laypunov exponents with each other.
\begin{proposition}\label{prop:Lyapunov-convergence-rate}
  Let $ I=[a,b] $, $ J=[b,c] $, $ ||I|-|J||\le \delta $. If
  $ (\eta,\omega,E)\in \R^d\times\DC_{|I|}\times\C $,
  is such that
  $ \norm{\eta}\le \rho(V) $, $ L_I(\eta,\omega,E)\ge \gamma>0 $, then we have
  \begin{equation}\label{eq:L_I-vs-L_J}
    |L_I(\eta,\omega,E)-L_{I\cup J}(\eta,\omega,E)|\le \frac{C(\log(1+|I|))^{1/\sigma}}{|I|}
  \end{equation}
  with $ C=C(V,d,\DC,|E|,\gamma,\delta) $. Furthermore, if
  $ \omega\in \DC $ then
  \begin{equation}\label{eq:L_I-vs-L}
    |L_I(\eta,\omega,E)-L(\eta,\omega,E)|\le \frac{C(\log(1+|I|))^{1/\sigma}}{|I|}
  \end{equation}
  with $ C=C(V,d,\DC,|E|,\gamma) $.
\end{proposition}

We will use the next estimate to see that positivity of the Lyapunov
exponent for some interval also implies positivity for smaller
intervals (we do this because we won't be able to apply
\cref{eq:L_I-vs-L} when $ \omega\in\DC_{|I|} $).  It is possible to
adjust the estimate to also give meaningful information when $ |I| $
is close to $ |J| $, but we are only interested in the case
$ |I|\gg |J| $.

\begin{lemma}\label{lem:Lyapunov-increasing}
  Let $ I=[a,b] $, $ J=[c,d] $, $ |I|>|J| $. Then for any
  $ (\eta,\omega,E)\in \R^d\times \T^d\times \C $,
  $ \norm{\eta}\le \rho(V) $ we have
  \begin{equation*}
    L_J(\eta,\omega,E)\ge L_I(\eta,\omega,E)-C(V,|E|)
    \left( \frac{|J|+1}{|I|-|J|} +\frac{1}{|J|} \right).
  \end{equation*}
\end{lemma}

Now we start the build-up toward the proof of the uniform upper bound
from \cref{prop:uniform-upper-bound}. This result is crucial for the
decay of the Green's function and the application of Cartan's
estimate. While for the decay of Green's function the simpler estimate
from \cref{cor:uniform-upper-bound} is enough, for Cartan's estimate
we also need the estimate to have good stability under (complex)
perturbations in $ (\theta,\omega,E) $. See \cite[Sec. 4]{GolSch08}
for the discrete counterparts of the results that follow.

One of the ingredients for the proof of
\cref{prop:uniform-upper-bound} will be the fact that the Lyapunov
exponent is Lipschitz with respect to $ \eta $.  This follows
immediately from the multivariable generalization of the following
fact from \cite{GolSch08}.
\begin{lemma}[{\cite[Lem. 4.1]{GolSch08}}]\label{lem:GS-Lipschitz-average}
  Let $ 1>\rho>0 $ and suppose $ u $ is subharmonic on
  \begin{equation*}
    A_\rho:=\{ z : 1-\rho<|z|<1+\rho   \}
  \end{equation*}
  such that $ \sup_{z\in A_\rho} u(z)\le 1 $ and
  $ \int_\T u(e(x))\,dx\ge 0 $ (we used the notation
  $ e(x)=e^{2\pi i x} $).  Then for any $ r_1,r_2 $ so that
  $ 1-\rho/2<r_1,r_2<1+\rho/2 $ one has
  \begin{equation*}
    \left| \int_\T u(r_1e(x))\,dx-\int_\T u(r_2 e(x))\,dx \right|\le C_\rho |r_1-r_2|. 
  \end{equation*}
\end{lemma}
The previous Lemma admits the following multivariable extension.
\begin{lemma}\label{lem:Lipschitz-average}
  Let $ 1>\rho>0 $ and suppose $ u $ is subharmonic in each variable
  on
  \begin{equation*}
    A_\rho^d:=\{ z\in \C^d : 1-\rho<|z_i|<1+\rho, i=1,\ldots,d   \}
  \end{equation*}
  such that
  \begin{equation*}
    0\le u(z)\le 1, \quad z\in A_\rho^d.
  \end{equation*}
  Then for any $ r,\tilde r\in \R^d $ so that
  $ 1-\rho/2<r_i,\tilde r_i<1+\rho/2 $, $ i=1,\ldots,d $, one has
  \begin{equation*}
    \left| \int_{\T^d} u(r_1 e(x_1),\ldots,r_d e(x_d))\,dx-\int_{\T^d} u(\tilde r_1 e(x_1),\ldots,\tilde r_d e(x_d))\,dx \right|
    \le C_\rho \sum_i|r_i-\tilde r_i|. 
  \end{equation*}
\end{lemma}
\begin{proof}
  The proof is by induction on $ d $. The case $ d=1 $ holds by
  \cref{lem:GS-Lipschitz-average}. We assume the result holds for
  $ d $ and prove it for $ d+1 $. Let
  \begin{equation*}
    v(z_{d+1})=\int_{\T^d} u(r_1e(x_1),\ldots,r_d e(x_d),z_{d+1})\,dx,
    \tilde v(z_{d+1})=\int_{\T^d} u(\tilde r_1e(x_1),\ldots,\tilde r_de(x_d),z_{d+1})\,dx.
  \end{equation*}
  By the induction assumption
  \begin{equation*}
    |v(z_{d+1})-\tilde v(z_{d+1})|\le C_\rho \sum_{i=1}^d |r_i-\tilde r_i|,\quad z_{d+1}\in A_\rho.
  \end{equation*}
  At the same time we have that $ v $ is subharmonic on $ A_\rho $ and
  $ 0\le v \le 1 $, so we can apply the case $ d=1 $ to it to get the
  desired conclusion. Indeed, we have
  \begin{multline*}
    \left|\int_\T v(r_{d+1}e(x_{d+1}))\,d x_{d+1}-\int_\T \tilde v(\tilde r_{d+1}e(x_{d+1}))\,dx_{d+1}\right|\\
    \le \left|\int_\T v(r_{d+1}e(x_{d+1}))\,dx_{d+1}-\int_\T v(\tilde r_{d+1}e(x_{d+1}))\,dx_{d+1}\right|\\
    +\left|\int_\T \left( v(\tilde r_{d+1}e(x_{d+1}))-\tilde v (\tilde r_{d+1}e(x_{d+1})) \right)\,dx_{d+1}\right|\\
    \le C_\rho |r_{d+1}-\tilde r_{d+1}|+ C_\rho\sum_{i=1}^d
    |r_i-\tilde r_i|.
  \end{multline*}
\end{proof}
As an immediate consequence of
\cref{lem:Lipschitz-average} we have the following result.
\begin{proposition}\label{prop:Lipschitz-Lyapunov}
  Let $ I=[a,b] $. If $ (\eta,\omega,E)\in \R^d\times \C^d\times \C $ are such that
  \begin{equation*}
    \max(|a|,|b|)\norm{\Im\omega}\le \rho(V),\quad \norm{\eta}\le \rho(V)
  \end{equation*}
  then
  \begin{equation*}
    |L_I(\eta,\omega,E)-L_I(\omega,E)|\le C \norm{\eta}, C=C(V,d,|E|).
  \end{equation*}
\end{proposition}

The other ingredient needed for \cref{prop:uniform-upper-bound} is the
stability of the logarithm of the norm of the transfer matrix. We have
the following ``rough'' estimate, that will be refined through the use
of the Avalanche Principle.

\begin{lemma}\label{lem:stability-rough}
  Let $ I=[a,b] $. Let $ (\theta_i,\omega_i,E_i)\in \C^d\times \C^d\times \C $,
  $ i=1,2 $, such that
  \begin{equation*}
    |E_2|\le|E_1|,\quad \norm{\Im \theta_i}\le \rho(V),\quad \max(|a|,|b|)\norm{\Im\omega_i}\le \rho(V).
  \end{equation*}
  Then we have
  \begin{multline*}
    \left| \log\norm{M_I(\theta_1,\omega_1,E_1)}-\log \norm{M_I(\theta_2,\omega_2,E_2)} \right|\\
    \le e^{C(V,|E_1|) |I|}
    \frac{\norm{\theta_1-\theta_2}+\max(|a|,|b|)\norm{\omega_1-\omega_2}+|E_1-E_2|}{\max_i\norm{M_I(\theta_i,\omega_i,E_i)}},
  \end{multline*}
  provided the right-hand side is $ \ll 1 $.
\end{lemma}
\begin{proof}
  Let $ Y_i(t) $, $ i=1,2 $, be solutions of the equations
  \begin{equation*}
    Y_i'(t)=A_i(t)Y_i(t),\quad
    A_i(t)= \begin{bmatrix}
      0 & 1 \\
      V(t,\theta_i+t\omega_i)-E_i & 0
    \end{bmatrix}.
  \end{equation*}
  Using the variations of constants method we have
  \begin{equation*}
    Y_2(b)=M_{[a,b]} Y_2(a)+\int_a^b M_{[a,s]}(A_2(s)-A_1(s))Y_2(s)\,ds
  \end{equation*}
  where $ M $ denotes the transfer matrix corresponding to the
  equation with $ i=1 $. So, if $ Y_1(a)=Y_2(a) $, then
  \begin{multline*}
    \norm{Y_2(b)-Y_1(b)}\le \int_a^b \norm{M_{[a,s]}}\norm{A_2(s)-A_1(s)}\norm{Y_2(s)}\,ds\\
    \le \norm{Y_2(a)} \int_a^b e^{C(V,|E_1|)(s-a)}(C(V)\norm{\theta_1-\theta_2}+C(V)|s|\norm{\omega_1-\omega_2}+|E_1-E_2|)\,ds\\
    \le \norm{Y_2(a)} e^{C
      |I|}(\norm{\theta_1-\theta_2}+\max(|a|,|b|)\norm{\omega_1-\omega_2}+|E_1-E_2|).
  \end{multline*}
  This implies 
  \begin{multline*}
    \norm{M_I(\theta_1,\omega_1,E_1)-M_I(\theta_2,\omega_2,E_2)}\\
    \le e^{C(V,|E_1|)
      |I|}(\norm{\theta_1-\theta_2}+\max(|a|,|b|)\norm{\omega_1-\omega_2}+|E_1-E_2|).
  \end{multline*}
  The conclusion follows from this and the fact that
  $ |\log x|\lesssim |x-1| $, provided $ |x-1|\ll 1 $.
\end{proof}

To refine the previous estimate, let us recall the Avalanche
Principle.
\begin{proposition}[{\cite[Prop. 2.2]{GolSch01}}]\label{prop:AP}
  Let $ A_1,\ldots,A_n $ be a sequence of $ 2\times 2
  $-matrices. Suppose that
  \begin{equation*}
    \min_{1\le j\le n} \norm{A_j}\ge \mu >n \text{ and}
  \end{equation*}
  \begin{equation*}
    \max_{1\le j<n} \left( \log\norm{A_{j+1}}+\log\norm{A_j}-\log \norm{A_{j+1}A_j} \right)<\frac{1}{2}\log \mu.
  \end{equation*}
  Then
  \begin{equation*}
    \left| \log\norm{A_n\ldots A_1}+\sum_{j=2}^{n-1}\log\norm{A_j}-\sum_{j=1}^{n-1}\log\norm{A_{j+1}A_j} \right|<C\frac{n}{\mu}.
  \end{equation*}
\end{proposition}

\begin{proposition}\label{prop:stability-AP}
  Let $ I=[a,b] $. Let $ (\eta_0,\omega_0,E_0)\in\R^d\times \DC_{|I|}\times \C $,
  such that $ \norm{\eta_0}\le \rho(V) $,
  $ L_I(\eta_0,\omega_0,E_0)\ge \gamma>0 $. Let
  \begin{equation}\label{eq:ell-lb}
    \ell \ge C(V,d,\DC,|E_0|,\gamma)(\log(1+|I|))^{1/\sigma}.
  \end{equation}
  There exists a set
  \begin{equation*}
    \cB=\cB_{I,\omega_0,E_0,\eta_0,\ell}, \mes(\cB)\le C\exp(-c\ell^{\sigma}),
    C=C(V,d,\DC,|E_0|,\gamma),c=c(V,d,\DC,|E_0|)
  \end{equation*}
  such that for any $ \theta_0\in\T^d\setminus \cB $ and any $ (\theta,\omega,E)\in \C^d\times\C^d\times \C $
  satisfying
  \begin{equation*}
    \norm{\theta-\theta_0-i\eta_0}+\max(|a|,|b|)\norm{\omega-\omega_0}+|E-E_0|\le \exp(-\ell^2),
  \end{equation*}
  we have
  \begin{equation*}
    |\log\norm{M_I(\theta_0+i\eta_0,\omega_0,E_0)}-\log\norm{M_I(\theta,\omega,E)}|\le \exp(-\gamma\ell/4).
  \end{equation*}
\end{proposition}
\begin{proof}
  Note that if $ \ell \ge |I|^{1/2+} $ then the estimate holds for all
  $ \theta_0 $ (provided $ |I| $ is large enough) by
  \cref{lem:stability-rough}. So it is enough to consider the case
  $ \ell < |I|^{1/2+} $.
  
  Partition $ I $ into $ n $ intervals $ J_i $ (ordered from left to
  right) of equal length and such that $ |J_i|\simeq \ell $.  Using
  the large deviations estimate and \cref{eq:L_I-vs-L_J} we can apply
  the Avalanche Principle with
  $ A_i=M_{J_i}(\theta_0+i\eta_0,\omega_0,E_0) $,
  $ \mu=\exp(\gamma \ell/2) $, provided $ \theta_0 $ is outside of the
  set $ \cB $ where the large deviations estimate fails on the
  intervals $ J_i $, $ J_i\cup J_{i+1} $. Note that we can apply
  \cref{eq:L_I-vs-L_J} because from \cref{lem:Lyapunov-increasing} it
  follows that $ L_{J_i}\ge 3\gamma/4 $ (here we used
  $ \ell < |I|^{1/2+} $). We clearly have
  \begin{equation*}
    \mes(\cB)\lesssim nC\exp(-c\ell^{\sigma})\le C\exp(-c\ell^{\sigma}/2).
  \end{equation*}
  We use the lower bound \cref{eq:ell-lb} on $ \ell $ for the above measure estimate
  and to ensure that $ \mu>n $.

  From the assumptions on $ \theta,\omega,E $ and
  \cref{lem:stability-rough} it follows that we can also apply the
  Avalanche Principle with $ \tilde A_i=M_{J_i}(\theta,\omega,E) $ and
  the same $ \mu $.  Subtracting the two Avalanche Principle
  expansions and applying \cref{lem:stability-rough} again we get
  \begin{multline*}
	|\log \norm{M_I(\theta_0+i\eta_0,\omega_0,E_0)}-\log \norm{M_I(\theta,\omega,E)}|\\
    \le \left|\sum \log\norm{A_i}-\log\norm{\tilde A_i} \right|
    +\left|\sum \log\norm{A_{i+1}A_i}-\log\norm{\tilde A_{i+1}\tilde
        A_i} \right|+
    C \frac{n}{\mu}\\
    \lesssim n\exp(C\ell)\exp(-\ell^2)+n\exp(-\gamma\ell/2)\le
    \exp(-\gamma\ell/4).
  \end{multline*}
\end{proof}

\begin{corollary}\label{cor:Lyapunov-stability}
  Let $ I=[a,b] $. Let $ (\eta_0,\omega_0,E_0)\in\R^d\times \DC_{|I|}\times \C $,
  such that $ \norm{\eta_0}\le \rho(V) $,
  $ L_I(\eta_0,\omega_0,E_0)\ge \gamma>0 $. Let
  \begin{equation*}
    \ell \ge C(V,d,\DC,|E_0|,\gamma)(\log(1+|I|))^{1/\sigma}.
  \end{equation*}
  Then we have
  \begin{equation*}
    |L_I(\eta_0,\omega_0,E_0)-L_I(\eta_0,\omega,E)|\le C\exp(-c\ell^{\sigma})
  \end{equation*}
  with $ C(V,d,\DC,|E_0|,\gamma) $, $ c=c(V,d,\DC,|E_0|) $, for any
  $ (\omega,E)\in \C^d\times \C $ such that
  \begin{equation*}
    \max(|a|,|b|)\norm{\omega-\omega_0}+|E-E_0|\le \exp(-\ell^2).
  \end{equation*}
\end{corollary}
\begin{proof}
  The conclusion follows by integrating the estimate from
  \cref{prop:stability-AP}.
\end{proof}

We are now ready to prove the uniform upper bound.
\begin{proposition}\label{prop:uniform-upper-bound}
  Let $ I=[a,b] $. Let $ (\omega_0,E_0)\in \DC_{|I|}\times \C $ be
  such that $ L_I(\omega_0,E_0)\ge \gamma>0 $. Then
  \begin{equation*}
    \sup_{\theta\in\T^d} \log\norm{M_I(\theta+\eta,\omega,E)}\le |I|L_I(\omega_0,E_0)+C|I|^{1-\sigma},
    C=C(V,d,\DC,|E_0|,\gamma)
  \end{equation*}
  for any $ (\eta,\omega,E)\in\C^d\times \C^d\times \C $ such that
  $ \norm{\eta}\le \rho(V)/(1+|I|) $ and
  \begin{equation}\label{eq:omega-E-restrictions}
    \max(|a|,|b|)\norm{\omega-\omega_0}+|E-E_0|\le \exp(-C(\log(1+|I|))^{2/\sigma}), C=C(V,d,\DC,|E_0|,\gamma) .
  \end{equation}
\end{proposition}
\begin{proof}
  Let $ \cB^{(1)}_{\eta}=\cB^{(1)}_{\eta,E_0,\omega_0} $ be the set
  from \cref{prop:stability-AP} with
  $ \ell = C(\log(1+|I|))^{1/\sigma} $ and
  $ C=C(V,d,\DC,|E_0|,\gamma) $ large enough so that
  $ \mes(\cB^{(1)}_\eta)\le 1/|I|^4 $. Let
  $ \cB^{(2)}_{\eta}=\cB^{(2)}_{\eta,E_0,\omega_0} $ be the
  exceptional set from \cref{thm:ldt}. We define
  \begin{equation*}
    \cB=\left\{\theta\in \C^d : \norm{\Im \theta}\le \rho(V),
      \Re\theta\in \cB^{(1)}_{\Im \theta}\cup\cB^{(2)}_{\Im \theta}\right\}.
  \end{equation*}
  We clearly have $ \mes(\cB)\le 2/|I|^4 $ and if
  $ \theta\in \C^d\setminus \cB $, $ \norm{\Im \theta}\le \rho(V)/(1+|I|) $,
  then by \cref{thm:ldt} and \cref{lem:Lipschitz-average}
  \begin{equation*}
    \log\norm{M_I(\theta,\omega_0,E_0)}\le |I|L_I(\Im\theta,\omega_0,E_0)+|I|^{1-\sigma}
    \le |I|L_I(\omega_0,E_0)+2|I|^{1-\sigma},
  \end{equation*}
  and by \cref{prop:stability-AP}
  \begin{equation}\label{eq:M_I-ub}
    \log\norm{M_I(\theta,\omega,E)}\le |I|L_I(\omega_0,E_0)+3|I|^{1-\sigma}
  \end{equation}
  for any $ \omega,E $ satisfying \cref{eq:omega-E-restrictions}.

  Let $ \theta_0\in \T^d $ arbitrary and $ \eta,\omega,E $ satisfying
  the needed assumptions. Let $ B_r $ be the ball centered at
  $ \theta_0+\eta $ and of radius $ r=1/|I|^2 $. Using the submean property of plurisubharmonic
  functions, \cref{eq:transfer-matrix-bound}, and \cref{eq:M_I-ub},  we have 
  \begin{multline*}
	\log\norm{M_I(\theta_0+\eta,\omega,E)}
    \le \frac{1}{\mes(B_r)} \int_{B_r} \log\norm{M_I(\theta,\omega,E)}\,d\theta\\
    \le \frac{1}{\mes(B_r)} \left( C\mes(\cB)+(\mes(B_r)-\mes(\cB))(|I|L_I(\omega_0,E_0)+3|I|^{1-\sigma})\right)\\
    \le |I|L_I(\omega_0,E_0)+C|I|^{1-\sigma}.
  \end{multline*}
  This concludes the proof.
\end{proof}

\begin{corollary}\label{cor:uniform-upper-bound}
  Let $ I=[a,b] $. If $ (\omega,E)\in \DC\times \C $ are such that
  $ L(\omega,E)\ge \gamma>0 $ then
  \begin{align*}
    \sup_{\theta\in\T^d} \log\norm{M_I(\theta,\omega,E)} &\le |I|L(\omega,E)+C|I|^{1-\sigma}\\
    \sup_{\theta\in\T^d} \log\norm{M_I(\theta,\omega,E)^{-1}} &\le |I|L(\omega,E)+C|I|^{1-\sigma}
  \end{align*}
  with $ C=C(V,d,\DC,|E|,\gamma) $.
\end{corollary}
\begin{proof}
  The first estimate follows from \cref{prop:uniform-upper-bound} and
  \cref{eq:L_I-vs-L}.

  Since $ \det M_I=1 $ we have
  \begin{equation*}
    \hsnorm{M_I^{-1}}=\hsnorm{M_I}\le \sqrt{2} \norm{M_I}
  \end{equation*}
  and the second estimate follows ($ \hsnorm{\cdot} $ denotes the
  Hilbert-Schmidt norm).
\end{proof}

\section{Decay of Green's Function}
\label{sec:decay}

We consider Green's function on a finite interval $ I=[a,b] $ with Dirichlet boundary conditions:
\begin{equation}\label{eq:Green-function}
  G_I(s,t)=G_I(s,t;\theta,\omega,E)=
  \begin{cases}
	\dfrac{v_a(s)v_b(t)}{W(v_a,v_b)} &, s\le t\\[1em]
    \dfrac{v_a(t)v_b(s)}{W(v_a,v_b)} &, t\le s
  \end{cases}.
\end{equation}
Recall that the functions $ v $ satisfy the initial conditions
\cref{eq:initial-conditions}.  

We will show that if a large deviations estimate holds on some
interval, then we get exponential decay for Green's function on
another interval of roughly the same size (this is similar to what
happens in the discrete case, see \cite[Prop. 7.19]{Bou05}). In fact,
due to Poisson's formula we will need this result for the partials of
Green's function.  Recall that for any solution $ y $ of
\cref{eq:eigenvalue-equation}, on an interval containing $ I $, the
Poisson formula reads
\begin{equation*}
  y(t)=y(b)\partial_s G_I(b,t)-y(a)\partial_s G_I(a,t).
\end{equation*}

\begin{proposition}\label{prop:Green-decay}
  Let $ I=[a,b] $. Let $ (\omega,E)\in \DC\times\C $ be such that
  $ L(\omega,E)\ge \gamma>0 $. If
  \begin{equation*}
    \log\norm{M_I(\theta,\omega,E)}\ge |I|L(\omega,E)-K,
  \end{equation*}
  with
  \begin{equation*}
    C(|I|^{1-\sigma}+1)\le K \le |I|^{1-}, C=C(V,d,\DC,|E|,\gamma)
  \end{equation*}
  (recall that $ \sigma $ is as in \cref{thm:ldt}), then there exists an interval $ J=J(\theta,\omega,E) $ such that
  \begin{gather*}
    J\subset I,\ |I|-|J|\le 4K/\gamma\\
    |G_J(s,t)|,|\partial_s G_J(s,t)|\le \exp(-|s-t|L(\omega,E)+2K),\
    s,t\in J.
  \end{gather*}
\end{proposition}
\begin{proof}
  By our assumption, at least one of the entries of $ M_I $ has to be
  $ \ge \frac{1}{2} \exp(|I|L-K) $. We treat each of the four
  possibilities separately.

  \noindent (1) Suppose
  \begin{equation*}
    |v_a(b)|\ge \frac{1}{2} \exp(|I|L-K).
  \end{equation*}
  In this case we let $ J=I $.  Using
  \cref{cor:uniform-upper-bound} we have
  \begin{multline*}
	|G_I(s,t)|= \left| \frac{v_a(s)v_b(t)}{v_a(b)} \right|
    \le 2 \exp((s-a)L+(b-t)L+C|I|^{1-\sigma}-(b-a)L+K)\\
    \le \exp(-(t-s)L+2K)
  \end{multline*}
  provided $ s\le t $ (it is enough to consider this case because
  $ G_I(s,t)=G_I(t,s) $).
  We used the fact that
  \begin{equation}\label{eq:u_a-u_b}
    W(v_a,v_b)=v_a(b)=-v_b(a).
  \end{equation} 
  The bound on $ |\partial_s G_I| $ is
  obtained in the same way, because the bounds from \cref{cor:uniform-upper-bound} apply to all the entries of the
  transfer matrix.

  \noindent (2) Suppose
  \begin{equation*}
    |v_a'(b)|\ge \frac{1}{2} \exp(|I|L-K).
  \end{equation*}
  For any $ t\in(a,b) $, there exists $ \tilde t\in(t,b) $ such that
  \begin{equation}
	|v_a'(b)-v_a'(t)|=|v_a''(\,\tilde t\,)(b-t)|
    =|v_a(\,\tilde t\,)(V(\tilde t,\theta+\tilde t \omega)-E)(b-t)|.
  \end{equation}
  Using \cref{cor:uniform-upper-bound} and choosing $ t $ so that
  $ b-t= 2K/\gamma $, it follows that
  \begin{multline*}
	|v_a(\,\tilde t\,)|\ge \frac{1}{C(V,|E|)(b-t)}|v_a'(b)-v_a'(t)|\\
    \ge \frac{1}{2C(b-t)}\exp(|I|L-K)
    \left[ 1-\exp((t-a)L+C|I|^{1-\sigma}-|I|L+K+\log 2) \right]\\
    \ge \frac{1}{2C(b-t)}\exp(|I|L-K)\left[ 1-\frac{1}{2}\exp(-(b-t)L+2K) \right]\\
    \ge \frac{1}{4C(b-t)}\exp(|I|L-K)\ge \exp(|I|L-3K/2).
  \end{multline*}
  The conclusion follows by the reasoning from case (1) applied to
  $ J=[a,\tilde t\,] $.
  
  \noindent (3) Suppose
  \begin{equation*}
    |u_a(b)|\ge \frac{1}{2} \exp(|I|L-K).
  \end{equation*}
  Note that we have
  \begin{equation*}
    W(u_a,v_b)=u_a(b)=v_b'(a).
  \end{equation*}
  Then, by the reasoning from case (2), there exists
  $ \tilde t\in I $, $ \tilde t-a\le 2K/\gamma $ such that
  \begin{equation*}
    |v_b(\,\tilde t\,)|\ge \exp(|I|L-3K/2).
  \end{equation*}
  Recall from \cref{eq:u_a-u_b} that we have
  $ |v_b(\,\tilde t\,)|=|v_{\,\tilde t\,}(b)| $ and so the conclusion
  follows by the argument from case (1) applied on
  $ J=[\,\tilde t,b] $.

  \noindent (4) Suppose
  \begin{equation*}
    |u_a'(b)|\ge \frac{1}{2} \exp(|I|L-K).
  \end{equation*}
  By the argument from case (2), there exists $ \tilde t $,
  $ b-\tilde t\le 2K/\gamma $ such that
  \begin{equation*}
    |u_a(\,\tilde t\,)|\ge \exp(|I|L-3K/2).
  \end{equation*}
  Following the reasoning from case (3) we get that there exists
  $ \bar t $, $ \bar t-a\le 2K/\gamma $ such that
  \begin{equation*}
    |v_{\,\bar t\,}(\,\tilde t\,)|\ge \exp(|I|L-5K/4).
  \end{equation*}
  The conclusion follows as in case (1) by taking
  $ J=[\,\bar t,\tilde t\,] $.
\end{proof}

Next we illustrate the well-known strategy of iterating Poisson's
formula to get the exponential decay of solutions, provided that we
have the decay of Green's function.

\begin{lemma}\label{lem:Poisson-iteration}
  Let $ 0<\ell \ll a\ll b $ and $ m>0 $ such that $ m\ell\gg 1 $.
  Suppose that for any $ t\in[a,b] $ there exists an interval
  $ J=[c_t,d_t] $, $ |J|\le \ell $, such that $ t\in J $, and 
  \begin{equation*}
    |\partial_s G_J(c_t,t)|,|\partial_s G_J(d_t,t)|\le \exp(-m\ell).
  \end{equation*}
  Then any solution $ y $ of \cref{eq:eigenvalue-equation} satisfies
  \begin{equation*}
    |y(t)|\le M \exp(-mt/8),\ t\in[2a,b/2], 
  \end{equation*}
  where $ M=\sup_{[a,b]}|y| $.
\end{lemma}
\begin{proof}
  For any $ t\in[2a,b/2] $ we can iterate Poisson's formula (on intervals $ J $ satisfying the assumptions)
  at least
  \begin{equation*}
    n=\min \left( [(t-a)/\ell],[(b-t)/\ell] \right)\ge \frac{t}{4\ell}
  \end{equation*}
  times to get
  \begin{equation*}
    |y(t)|\le M 2^n \exp(-nm\ell)\le M\exp(-nm\ell/2)\le M\exp(-mt/8).
  \end{equation*}
\end{proof}

\section{Cartan Sets}\label{sec:Cartan}

We will use $ D(z,r) $ to denote the disk of radius $ r $ centered at
$ z\in \C $.
\begin{definition} Let $H \gg 1$.  For an arbitrary subset
  $\cB \subset D(z_0, 1)\subset \C$ we say that $\cB \in \Car_1(H, K)$
  if $\cB\subset \bigcup\limits^{j_0}_{j=1} D(z_j, r_j)$ with
  $j_0 \le K$, and
  \begin{equation}
    \sum_j\, r_j < e^{-H}\ .
  \end{equation}
  If $d$ is a positive integer greater than one and
  $\cB \subset \prod\limits_{i=1}^d D(z_{i,0}, 1)\subset \C^d$ then we
  define inductively that $\cB\in \Car_d(H, K)$ if for any
  $ J\subset \{1,\ldots,d\} $, $ |J|<d $, there exists
  \begin{equation*}
    \cB_J \subset \prod_{j\in J} D(z_{j,0}, 1) \subset \C^{|J|},\ \cB_J \in \Car_{|J|}(H, K)
  \end{equation*}
  so that $\cB_{J'}(z) \in \Car_{|J'|}(H, K)$ for any
  $z \in \C^{|J|} \setminus \cB_J$, where
  \begin{equation*}
    \cB_{J'}(z) = \{ w_{J'} : w\in \cB, w_J=z   \}.
  \end{equation*}
  We used $ J' $ to denote $ \{ 1,\ldots,d \}\setminus J $ and given
  $ z\in \C^d $, $ z_J $ denotes the vector $ (z_j)_{j\in J} $.
\end{definition}

The above definition is a simple extension of
\cite[Def. 2.12]{GolSch08}, where only the case $ |J|=1 $ is
considered.  The reason behind the definition of Cartan sets is the
following result, referred to as the Cartan estimate. The Cartan
estimate from \cite{GolSch08} holds even with this slightly more
general definition. The proof is essentially the same, one only needs
to use complete induction instead of the regular induction used in
\cite{GolSch08}.
\begin{lemma}[{\cite[Lem. 2.15]{GolSch08}}]\label{lem:Cartan-estimate}
  Let $\varphi(z_1, \dots, z_d)$ be an analytic function defined in a
  polydisk $P = \prod\limits^d_{j=1} D(z_{j,0}, 1)$, $z_{j,0} \in \C$.
  Let $M \ge \sup\limits_{z\in P} \log |\varphi(z)|$,
  $m \le \log \bigl |\varphi(z_0)\bigr |$,
  $z_0 = (z_{1,0},\dots, z_{d,0})$.  Given $H\gg 1$, there exists a
  set $\cB \subset P$, $\cB \in \Car_d\left(H^{1/d}, K\right)$,
  $K = C_d H(M - m)$, such that
  \begin{equation}\label{eq:cart_bd}
    \log \bigl | \varphi(z)\bigr | > M-C_d H(M-m)
  \end{equation}
  for any $z \in \prod^d_{j=1} D(z_{j,0}, 1/6)\setminus \cB$.
\end{lemma}

Let us note that the definition of the Cartan sets gives information
about their measure.
\begin{lemma}\label{lem:Cartan-measure}
  For any $\cB \subset \prod\limits_{i=1}^d D(z_{i,0}, 1)\subset \C^d$
  such that $ \cB\in \Car_d(H,K) $ we have
  \begin{equation*}
    \mes_{\C^d}(\cB)\lesssim d e^{-H} \text{ and } \mes_{\R^d}(\cB\cap \R^d)\lesssim d e^{-H}.
  \end{equation*}
\end{lemma}
\begin{proof}
  The case $ d=1 $ follows immediately from the definition of
  $ \Car_1 $. The case $ d>1 $ follows by induction, using Fubini and
  the definition of $ \Car_d $.
\end{proof}

We use Cartan's estimate to argue that if the large deviations
estimate fails then for fixed phase and frequency the energy must be
in a finite union of small intervals, with a good bound on the number
of intervals. This is only possible up to some small exceptional sets
of phases and frequencies. To be able to apply Cartan's estimate
effectively we need to restrict ourselves to the case when the
Lyapunov exponent is positive, so that we have the uniform upper bound
from \cref{prop:uniform-upper-bound}.
\begin{proposition}\label{prop:Cartan-E}
  Let $ I=[a,b] $. Let $ [E',E'']\subset \R $, $ \gamma>0 $ and
  \begin{equation*}
    \cP_I= \{ (\omega,E)\in \DC_{|I|}\times [E',E'']: L_I(\omega,E)\ge \gamma   \}.
  \end{equation*}
  Let
  \begin{equation*}
    H\ge C(\log(1+|I|))^A,\ C=C(V,d,\DC,E',E'',\gamma),\  A=A(\sigma,d).
  \end{equation*}
  If $ |I|\ge C(V,d,\DC,E',E'',\gamma) $ then there exists
  \begin{equation*}
    \Theta_I \subset \T^d,\ \mes(\Theta_I)\le \max(|a|,|b|)^d \exp(-H^{1/(2d+1)}/2)
  \end{equation*}
  such that if $ \theta\in \T^d\setminus \Theta_I $ and
  $ (\omega,E)\in \cP_I $ are such that
  \begin{equation*}
    \log\norm{M_I(\theta,\omega,E)}\le |I|L_I(\omega,E)-CH|I|^{1-\sigma},\ C=C(V,d,\DC,E',E'',\gamma)
  \end{equation*}
  then either $ \omega\in \Omega_{I,\theta} $ or
  $ E\in \cE_{I,\theta,\omega} $, where
  $ \mes(\Omega_{I,\theta})\le \exp(-H^{1/(2d+1)}/2) $ and
  $ \cE_{I,\theta,\omega} $ is the union of less than
  \begin{equation*}
    H\exp(C(\log(1+|I|))^{2/\sigma}),\ C=C(V,d,\DC,E',E'',\gamma)
  \end{equation*}
  intervals, each of measure less than $ \exp(-H^{1/(2d+1)}) $.
\end{proposition}
\begin{proof}
  Let $ r=\exp(-C(\log(1+|I|))^{2/\sigma}) $ and let
  \begin{gather*}
	P(\theta_j,r),\ 1\le j \lesssim r^{-d}\\
    P(\omega_k,r/\max(|a|,|b|)),\ 1\le k \lesssim r^{-d}\max(|a|,|b|)^d\\
    D(E_l,r),\ 1\le l \lesssim |E''-E'|,\ 1\le l \lesssim r^{-1}
    |E''-E'|
  \end{gather*}
  be covers of $ \T^d $, $ \DC_{|I|} $, and $ [E',E''] $ respectively
  ($ P(\cdot,r) $ denotes a polydisk of radius $ r $).  Let
  \begin{equation*}
    \cI= \{ (j,k,l):  (P(\theta_j,r)\times P(\omega_k,r/\max(|a|,|b|)) \times D(E_l,r))\cap (\T^d\times \cP_I)\neq \emptyset \}.
  \end{equation*}

  Let $ \iota=(j,k,l)\in \cI $ and
  \begin{equation*}
    P_\iota(r)=P(\theta_j,r)\times P(\omega_k,r/\max(|a|,|b|)) \times D(E_l,r).
  \end{equation*}
  By the definition of $ \cI $ and \cref{thm:ldt} there exists
  $ (\theta_\iota, \omega_\iota,E_\iota)\in P_\iota(r)\cap (\T^d\times
  \cP_I) $ such that
  \begin{equation*}
    \log\norm{M_I(\theta_\iota,\omega_\iota,E_\iota)}\ge |I|L_I(\omega_\iota,E_\iota)-|I|^{1-\sigma}.
  \end{equation*}
  Let
  \begin{equation*}
    \tilde P_\iota(r)=P(\theta_\iota,r)\times P(\omega_\iota,r/\max(|a|,|b|)) \times D(E_\iota,r).
  \end{equation*}
  Since $ L_I(\omega_\iota,E_\iota)\ge \gamma $ we can use
  \cref{prop:uniform-upper-bound} to guarantee that
  \begin{equation*}
    \sup \{  \log\norm{M_I(\theta,\omega,E)} : (\theta,\omega,E)\in \tilde P_\iota (12r) \}
    \le |I|L_I(\omega_\iota,E_\iota)+C|I|^{1-\sigma}
  \end{equation*}
  with $ C=C(V,d,\DC,E',E'',\gamma) $ (we just need to take $ C $ from
  the definition of $ r $ to be large enough).

  Next we set things up to apply \cref{lem:Cartan-estimate}. Let
  \begin{gather*}
    f_I=u_a(b)^2+u_a'(b)^2+v_a(b)^2+v_a'(b)^2\\
    S_\iota(\theta,\omega,E)=(\theta_\iota+12 r\theta,\omega_\iota+12 r\omega/\max(|a|,|b|),E_\iota+12 rE)\\
    \phi_\iota(\theta,\omega,E)=f_I(S_\iota(\theta,\omega,E)),\
    (\theta,\omega,E)\in D(0,1)^{2d+1}.
  \end{gather*}
  Note that we have
  \begin{equation*}
    |f_I|\le \hsnorm{M_I}^2\le 2\norm{M_I}^2
  \end{equation*}
  and if $ (\theta,\omega,E)\in\R^{2d+1} $, then
  \begin{equation*}
    |f_I(\theta,\omega,E)|= \hsnorm{M_I(\theta,\omega,E)}^2\ge \norm{M_I(\theta,\omega,E)}^2.
  \end{equation*}
  and therefore
  \begin{gather*}
    \log|\phi_\iota(\theta_\iota,\omega_\iota,E_\iota)|\ge 2(|I|L_I(\omega_\iota,E_\iota)-|I|^{1-\sigma}),\\
    \ \sup_{D(0,1)^{2d+1}} \log|\phi_\iota(\theta,\omega,E)|\le
    2(|I|L_I(\omega_\iota,E_\iota)+C|I|^{1-\sigma}).
  \end{gather*}
  By applying Cartan's estimate to $ \phi_\iota $ and by using
  \cref{cor:Lyapunov-stability} we get that
  \begin{equation*}
    \log\norm{M_I(\theta,\omega,E)}\ge |I|L_I(\omega_\iota,E_\iota)-CH|I|^{1-\sigma}
    \ge |I|L_I(\omega,E)-2CH|I|^{1-\sigma}
  \end{equation*}
  for all
  \begin{equation*}
    (\theta,\omega,E)\in P_\iota(r)\setminus S_\iota(\cB_\iota)\subset \tilde P_\iota(2r)\setminus S_\iota(\cB_\iota)
  \end{equation*}
  with $ \cB_\iota\in \Car_{2d+1}(H^{1/(2d+1)},K) $,
  $ K=CH|I|^{1-\sigma} $.

  From the definition of Cartan sets we know that there exists a set
  $ \Theta_\iota\in \Car_d(H^{1/(2d+1)},K) $ such that if
  $ \theta\notin \Theta_\iota $ then
  $ \cB_\iota(\theta)\in \Car_{d+1}(H^{1/(2d+1)},K) $, with
  \begin{equation*}
    \cB_\iota(\theta)= \{ (\omega,E): (\theta,\omega,E)\in\cB_\iota  \}.
  \end{equation*}
  Applying the definition again we see that if
  $ (\omega,E)\in \cB_\iota(\theta) $ then either
  \begin{equation*}
    \omega\in \Omega_{\iota,\theta},\  \Omega_{\iota,\theta} \in \Car_d(H^{1/(2d+1)},K)
  \end{equation*}
  or
  \begin{equation*}
    E\in \cE_{\iota,\theta,\omega}:=\cB_\iota(\theta,\omega),\ \cE_{\iota,\theta,\omega}\in \Car_1(H^{1/(2d+1)},K).
  \end{equation*}
  Note that by the definition of $ \Car_1 $ we have that
  $ \cE_{\iota,\theta,\omega}\cap\R $ is contained in the union of at
  most $ K $ intervals, each of measure smaller than
  $ \exp(-H^{1/(2d+1)}) $.

  Let
  \begin{gather*}
	j_0=j_0(\theta)= \min \{j: \theta\in P(\theta_j,r)\}\\
    k_0=k_0(\theta,\omega)=\min \{k: (\theta,\omega)\in
    P(\theta_{j_0},r)\times P(\omega_k,r/\max(|a|,|b|))\}.
  \end{gather*}
  Now the conclusion follows by setting
  \begin{gather*}
    \Theta_I= \cup \{ (\theta_j+12r \Theta_\iota)\cap P(\theta_j,r)
    :\iota=(j,k,l)\in \cI \}\cap\R^d,\\
    \Omega_{I,\theta}= \cup
    \{(\omega_k+12r/\max(|a|,|b|)\Omega_{\iota,\theta})\cap
    P(\omega_k,r/\max(|a|,|b|))
    : \iota=(j_0,k,l)\in \cI \}\cap \R^d,\\
    \cE_{I,\theta,\omega}= \cup
    \{(E_l+12r\cE_{\iota,\theta,\omega})\cap D(E_l,r) :
    \iota=(j_0,k_0,l)\in \cI\}\cap \R.
  \end{gather*}
  Note that for the measure estimates we use
  \cref{lem:Cartan-measure}.
\end{proof}

\section{Semialgebraic Sets}\label{sec:semialgebraic}

Recall that a set $ \cS\subset\R^n $ is called semialgebraic if it is
a finite union of sets defined by a finite number of polynomial
equalities and inequalities. So, a (closed) semialgebraic set is given by an expression
\begin{equation*}
  \cS=\cup_j \cap_{\ell\in L_j} \{ P_\ell s_{j\ell} 0  \},
\end{equation*}
where $ \{P_1,\ldots,P_s\} $ is a collection of polynomials of $ n $
variables, $ L_j\subset\{1,\ldots,s\} $ and
$ s_{jl}\in\{\ge,\le,=\} $. If the degrees of the polynomials are
bounded by $ d $ then we say that the degree of $ \cS $ is bounded by
$ sd $. We refer to \cite[Chap. 9]{Bou05} for more information on
semialgebraic sets.

The main result of this section is Lemma \ref{lem:semi-algebraic-all},
in which we argue that the set of $ (\theta,\omega,E) $ for which the
large deviations estimate fails is contained in a semialgebraic set of
controlled size and degree. To this end we will need to approximate
the entries of $ M_I $ by polynomials of controlled degree. Since
$ V $ is complex analytic in a neighborhood of $ \H_\rho^{d+1} $, $ \rho=\rho(V) $ (recall \cref{eq:H_rho}),
 and periodic in the real direction it follows that
any entry of $ M_I $ is also complex analytic in a neighborhood of
\begin{equation}\label{eq:rho'}
  \H_{\rho'}^{2d}\times \C,\ \rho'= \frac{\rho}{2(1+\max(|a|,|b|))},
\end{equation}
and periodic in the real direction for the phase variables
(we chose $ \rho' $ such that
$ \theta+t\omega\in \H_{\rho}^d $ for $ t\in I $). We can use Fourier series and Taylor series to obtain a
polynomial approximation in the phase and energy variables, but not in the frequency variables. One could use
Taylor series for the frequency variables, but only at the cost of getting different approximating polynomials on
different frequency intervals. We avoid this inconvenience by using Faber series.

We recall the basic information we will need about Faber polynomials and Faber series. We refer to
\cite[Ch. 2,3]{Sue98} for further information (see also \cite[Sec. 3.14-15]{Mar67}).
Let $ K\subset \C $ be a compact set such that
its complement is simply connected (on the Riemann sphere). Let $ \varphi_K $ be the conformal mapping of the
complement of $ K $ onto the complement of the unit disk, normalized such that
$ \varphi_K(\infty)=\infty $ and $ \varphi_K'(\infty)>0 $. Faber's polynomials, denoted by $ \Phi_{K,n} $,
$ n\ge 0 $,
are the polynomial parts of the Laurent series expansion of $ \varphi_K^n $ at $ \infty $. It is clear from the
definition that $ \Phi_{K,n} $ has degree $ n $. Given $ R>1 $ we let
$ \Gamma_{K,R}=\varphi_K^{-1}(\{ |z|=R \}) $ and
we denote by $ G_{K,R} $ the bounded domain enclosed by $ \Gamma_{K,R} $. It can be seen that
\begin{equation}\label{eq:Faber-polynomial}
  \Phi_{K,n}(z)=\frac{1}{2\pi i}\int_{\Gamma_{K,R}} \frac{\varphi^n_K(\zeta)}{\zeta-z}\,d\zeta,\ z\in G_{K,R}.
\end{equation}
If $ f $ is an analytic function on $ G_{K,R} $, then it can be expanded in a series with respect to the
Faber polynomials
\begin{equation*}
  f(z)=\sum_{n=0}^\infty a_n\Phi_{K,n}(z),\ z\in G_{K,R}
\end{equation*}
which converges absolutely and locally uniformly in $ G_{K,R} $. The Faber coefficients  are given by
\begin{equation*}
  a_n=\frac{1}{2\pi i}\int_{|t|=\rho} \frac{f(\varphi_K^{-1}(t))}{t^{n+1}}\,dt
\end{equation*}
for any $ \rho\in(1,R) $.

We are now ready to state our abstract approximation result for functions which are analytic on product
sets.

\begin{proposition}\label{prop:Faber-approximation}
  Let $ K_1,\ldots,K_m $ be compact sets in $ \C $ such that their complements are simply connected (on the
  Riemann sphere). Let $ R>1 $ and let $ f $ be an analytic function on a neighborhood of the closure of
  $ G_{K_1,R}\times\ldots\times G_{K_m,R} $.
  Given $ N\ge 0 $ and $ R'\in(1,R) $, there exists a polynomial $ P_N $ of degree at most $ N $ such that
  \begin{multline*}
    \sup \{ |f(z)-P_N(z)| : z\in K_1\times\ldots\times K_m \}\\
    \le C(m) \left( \frac{R'}{R} \right)^N
    \left( \prod_{i=1}^m \frac{\ell(\Gamma_{K_i,R'})}{d(K_i,\Gamma_{K_i,R'})} \right)
    \sup \{|f(z)|: z\in G_{K_1,R}\times\ldots\times G_{K_m,R}\},
  \end{multline*}
  where $ \ell(\Gamma_{K_i,R'}) $ denotes the length of $ \Gamma_{K_i,R'} $ and $ d(K_i,\Gamma_{K_i,R'}) $
  denotes the distance between $ K_i $ and $ \Gamma_{K_i,R'} $.
\end{proposition}
\begin{proof}
  If we take the Faber series expansion of $ f $ with respect to one of its variables, it is clear that the
  coefficients will be analytic with respect to the other variables. Since the Faber series converges absolutely
  we obtain through iteration the following expansion for $ f $ on
  $ G_{K_1,R}\times\ldots\times G_{K_m,R} $
  \begin{equation*}
    f(z_1,\ldots,z_m)= \sum_{ n} a_{n} \Phi_{K_1,n_1}(z_1)\ldots \Phi_{K_m,n_m}(z_m),
  \end{equation*}
  with the coefficients given by
  \begin{equation*}
    a_n= \frac{1}{(2\pi i)^m}
    \int_{|t_m|=R}\ldots \int_{|t_1|=R}
    \frac{f(\varphi_{K_1}^{-1}(t_1),\ldots,\varphi_{K_m}^{-1}(t_m))}
    {t_1^{n_1+1}\ldots t_m^{n_m+1}}\,dt_1\ldots dt_m.
  \end{equation*}
  Note that we have
  \begin{equation*}
    |a_n|\le \frac{1}{(2\pi)^m R^{|n|}} \sup \{|f(z)|: z\in G_{K_1,R}\times\ldots\times G_{K_m,R}\},
  \end{equation*}
  where $ |n|=n_1+\ldots+n_m $. Also, from \cref{eq:Faber-polynomial} (with $ R' $ instead of $ R $) it follows
  that
  \begin{equation*}
    \sup \{|\Phi_{K_i,n_i}(z)| : z\in K_i\} \le
    \frac{(R')^{n_i} \ell(\Gamma_{K_i,R'})}{2\pi d(K_i,\Gamma_{K_i,R'})}.
  \end{equation*}
  Therefore the conclusion holds by taking
  \begin{equation*}
    P_N(z_1,\ldots,z_m)=\sum_{ |n|\le N} a_{n} \Phi_{K_1,n_1}(z_1)\ldots \Phi_{K_m,n_m}(z_m).
  \end{equation*}
\end{proof}

\begin{remark}
  The previous proposition is a more explicit version of the direct statement of the so called Bernstein-Walsh
  theorem. Such results are also known for functions which are not necessarily defined on a product set,
  see \cite{Sic81}, \cite{Lev06}, but their statements are not explicit enough for our purposes. 
\end{remark}

\begin{lemma}\label{lem:approximation}
  Let $ I=[a,b] $, $ T\ge 0 $, $ [E',E'']\subset \R $. There exists a
  polynomial $ P_I(\theta,\omega,E) $ of degree less than
  \begin{equation}\label{eq:degree-bound}
    C [(1+\max(|a|,|b|)) (1+|I|)(1+T)]^2,\ C=C(V,d,E',E'') 
  \end{equation}
  such that
  \begin{equation*}
    \left|\log\norm{M_I(\theta+t\omega,\omega,E)}-\frac{1}{2}\log|P_I(\theta+t\omega,\omega,E)|\right|\lesssim 1, 
  \end{equation*}
  for any $ (\theta,\omega,E)\in \T^d\times\T^d\times [E',E''] $ and
  $ |t|\le T $.
\end{lemma}
\begin{proof}
  Let $ f(\theta,\omega,E) $ denote one of the entries of
  $ M_I(\theta,\omega,E) $. We already noted that $ f $ is analytic on
  $ \H_{\rho'}^{2d}\times \C $ (see \cref{eq:rho'}). Let $ K_i=[-L_i,L_i] $, where 
  \begin{equation*}
    L_i=\begin{cases}
      1+T &, i=1,\ldots, d\\
      1 &, i=d+1,\ldots,2d\\
      \max(|E'|,|E''|) &, i=2d+1 
    \end{cases}.
  \end{equation*}
  We want to apply \cref{prop:Faber-approximation} to approximate $ f $ by a polynomial on $ \prod_i K_i $.
  The mappings needed for \cref{prop:Faber-approximation} are scaled versions of the Zhukowsky transform:
  \begin{equation*}
    \varphi_{K_i}(z)= \frac{z}{L_i}+ \sqrt{\left( \frac{z}{L_i} \right)^2-1},\quad
    \varphi_{K_i}^{-1}(w)=\frac{L_i}{2}\left( w+\frac{1}{w} \right).
  \end{equation*}
  If we let $ R=1+\epsilon $, with $ \epsilon\ll \rho'/(\max_i L_i) $, then
  $ \prod_i G_{K_i,R} \subset  \H_{\rho'}^{2d}\times \C $. We choose $ R'=(1+R)/2 $. It is elementary to see
  that
  \begin{equation*}
    \ell(\Gamma_{K_i,R'})\le \pi L_i \left( R'+ \frac{1}{R'} \right),\quad
    d(\Gamma_{K_i,R'}, K_i)=\frac{L_i}{2} \left( R'+\frac{1}{R'}-2 \right).
  \end{equation*}
  By \cref{prop:Faber-approximation}, for any $ N\ge 0 $,
  there exists a polynomial $ P_N $ of degree less than $ N $ such that
  on $ \prod_i K_i $ we have
  \begin{multline*}
	|f-P_N|\le C(d) \left( \frac{R'}{R} \right)^N
    \left( \frac{(R')^2+1}{(R'-1)^2} \right)^{2d+1} \norm{f}_\infty\\
    \le C\exp(-cN\epsilon)\exp(-2(2d+1)\log \epsilon)\exp(C|I|),
  \end{multline*}
  provided $ \epsilon\ll 1 $. Clearly, if we choose $ N $ as in \cref{eq:degree-bound} we get that
  $ |f-P_N|\le \exp(-cN\epsilon/2) $.
  
  By approximating each entry of $ M_I $ in this way we obtain a $ 2\times 2 $ matrix with polynomial entries with
  the desired degree bounds and such that
  \begin{equation*}
    \norm{M_I(\theta+t\omega,\omega,E)-\tilde M_I(\theta+t\omega,\omega,E)}\ll 1
  \end{equation*}
  for any $ (\theta,\omega,E)\in \T^d\times\T^d\times [E',E''] $ and
  $ |t|\le T $. Therefore, the conclusion holds with
  $ P_I=\hsnorm{\tilde M_I}^2 $.
  
\end{proof}

We will also need a way to approximate the Lyapunov exponent
$ L_I(\omega,E) $. We use the same strategy of averaging over the
phase shifts of $ \log\norm{M_I} $ as in \cite[Lem. 9.1]{BouGol00},
but we give a different proof.  We base our proof on the following
fact.

\begin{lemma}[{\cite[Cor. 9.7]{Bou05}}]\label{lem:sublinear-count}
  Let $ \cS\subset [0,1]^d $ be semialgebraic of degree $ B $ and
  $ \mes \cS<\eta $. Let $ N $ be an integer such that
  \begin{equation*}
    \log B \ll \log N < \log \frac{1}{\eta}.
  \end{equation*}
  Then, for any $ \theta_0\in\T^d $, $ \omega\in\DC_N $
  \begin{equation*}
    \# \{ n=1,\ldots,N : \theta_0+n\omega\in \cS (\!\!\!\!\!\mod 1)  \}<N^{1-\delta}
  \end{equation*}
  for some $ \delta=\delta(\DC) $.
\end{lemma}

In order to apply \cref{lem:sublinear-count} we will need a
semialgebraic approximation of the set where the large deviations
estimate fails, but only in the phase variable.
\begin{lemma}\label{lem:semi-algebraic-phase}
  Let $ I=[a,b] $, $ \omega\in \T^d $, $ E\in [E',E''] $, and
  \begin{equation*}
    \cB_I(H)=\cB_I(H,\omega,E) := \{ \theta\in \T^d :
    \log\norm{M_I(\theta,\omega,E)}\le |I|L_I(\omega,E)-H \}.
  \end{equation*}
  There exists a semialgebraic set $ \cS_I(H)=\cS_I(H,\omega,E) $ of
  degree less than
  \begin{equation*}
    C(1+\max(|a|,|b|))(1+|I|),\ C=C(V,d,E',E'')
  \end{equation*}
  such that
  \begin{equation*}
    \cB_I(H)\subset \cS_I(H) \subset \cB_I(H/2),
  \end{equation*}
  provided $ H\gg 1 $.
\end{lemma}
\begin{proof}
  Let $ P_I $ be the polynomial from \cref{lem:approximation} with
  $ T=1 $. Then
  \begin{equation*}
    \left|\log\norm{M_I(\theta,\omega,E)}-\frac{1}{2}\log|P_I(\theta,\omega,E)|\right|\le C_0
  \end{equation*}
  and the conclusion follows by taking
  \begin{equation*}
    \cS_I(H)= \left\{ \theta:  \frac{1}{2}\log|P_I(\theta,\omega,E)|\le |I|L_I(\omega,E)-H+C_0 \right\}.
  \end{equation*}
\end{proof}

Now we can prove the estimate that will let us approximate $ L_I $.

\begin{lemma}\label{lem:averaging}
  Let $ I=[a,b] $, $ \omega\in \DC_{|I|} $, $ E\in [E',E''] $ such that $ L_I(\omega,E)\ge \gamma>0 $.  If
  $ |I|\ge C(V,d,\DC,E',E'') $ then
  \begin{equation*}
    \left| \frac{1}{N}\sum_{n=1}^N\log \norm{M_I(\theta+n\omega,\omega,E)}-|I|L_I(\omega,E) \right|
    \le C |I|^{1-\sigma},\  C=C(V,d,\DC,E',E'',\gamma)
  \end{equation*}
  for any integer $ N $ such that
  \begin{equation*}
    C(V,d,\DC,E',E'')\log \max(|a|,|b|,|I|)\le \log N < c(V,d,\DC,E',E'')|I|^{1-\sigma}.
  \end{equation*}
\end{lemma}
\begin{proof}
  Let
  \begin{equation*}
    \cB= \{ \theta\in \T^d : \log \norm{M_I(\theta,\omega,E)} \le |I|L_I(\omega,E) - |I|^{1-\sigma}  \}.
  \end{equation*}
  By Lemma \ref{lem:semi-algebraic-phase} there exists a
  semi-algebraic set $ \cS $ such that
  \begin{equation*}
    \cB\subset \cS,\quad \deg \cS \le C\max(|a|,|b|)|I|, \quad \mes \cS \le \exp(-c|I|^{1-\sigma}),
  \end{equation*}
  provided $ |I|\ge C(V,d,\DC,E',E'') $.

  For any $ \theta\in \T^d \setminus \cS $ we have 
  \begin{equation*}
    |I|L_I(\omega,E)-|I|^{1-\sigma}\le \log \norm{M_I(\theta,\omega,E)} \le |I|L_I(\omega,E)+C|I|^{1-\sigma},
  \end{equation*}
  (recall the upper bound from \cref{prop:uniform-upper-bound})
  whereas for $ \theta\in \cS $ we have
  \begin{equation*}
    0\le \log \norm{M_I(\theta,\omega,E)} \le |I|L_I(\omega,E)+C|I|^{1-\sigma}.
  \end{equation*}
  From the above and Lemma \ref{lem:sublinear-count} we get
  \begin{multline*}
    C|I|^{1-\sigma}
    \ge \frac{1}{N} \sum_{n=1}^N \log \norm{M_I(\theta+n\omega,\omega,E)}-|I|L_I(\omega,E) \\
    \ge \frac{N-N^{1-\delta}}{N}(-|I|^{1-\sigma})+
    \frac{N^{1-\delta}}{N}(-|I|L_I(\omega,E)) \ge -2|I|^{1-\sigma}
  \end{multline*}
  provided
  \begin{equation*}
    C(V,d,\DC,E',E'')\log \max(|a|,|b|,|I|)\le \log N < c(V,d,\DC,E',E'')|I|^{1-\sigma}.
  \end{equation*}
  This concludes the proof.
\end{proof}

\begin{lemma}\label{lem:semi-algebraic-all}
  Let $ I=[a,b] $, $ [E',E'']\subset\R $, $ \gamma>0 $.  If
  $ |I|\ge C(V,d,\DC,E',E'',\gamma) $ and
  \begin{gather*}
    T=(\max(|a|,|b|,|I|))^C,\ C=C(V,d,\DC,E',E'')\\
    \begin{aligned}
      \cB_I(H,\gamma):= \{ (\theta,\omega,E)\in \T^d\times \DC_T\times
      [E',E''] &:&
      \log\norm{M_I(\theta,\omega,E)}\le |I|L_I(\omega,E)-H,\\
      & & L_I(\omega,E)\ge \gamma \}
    \end{aligned}
  \end{gather*}
  then there exists a semi-algebraic set $ \cS_I=\cS_I(H,\gamma) $ of
  degree less than $ T^{C(d)} $ such that
  \begin{equation*}
    \cB_I(H,\gamma)\subset \cS_I(H,\gamma) \subset \cB_I(H/2,\gamma/2),
  \end{equation*}
  provided
  \begin{equation*}
    H\ge C|I|^{1-\sigma},\ C=C(V,d,\DC,E',E'').
  \end{equation*}
\end{lemma}
\begin{proof}
  Let $ P_I $ be the polynomial from \cref{lem:approximation}. Then
  \begin{equation*}
    \left|\log\norm{M_I(\theta+t\omega,\omega,E)}-\frac{1}{2}\log|P_I(\theta+t\omega,\omega,E)|\right|\le C_0,
    \ |t|\le T
  \end{equation*}
  and the degree of $ P_I $ is less than $ T^C $.  If furthermore, the
  power in the definition of $ T $ is large enough so that we can
  apply \cref{lem:averaging} with $ N=[T] $, then the conclusion
  follows by taking $ \cS_I(H,\gamma) $ to be the set of
  $ (\theta,\omega,E)\in \T^d\times \DC_T\times [E',E''] $ such that
  \begin{gather*}
	\frac{1}{2}\log|P_I(\theta,\omega,E)|
    \le \frac{1}{N}\sum_{n=1}^N \frac{1}{2}\log|P_I(\theta+n\omega,\omega,E)|+2C_0+C|I|^{1-\sigma}-H\\
    \gamma|I|\le \frac{1}{N}\sum_{n=1}^N
    \frac{1}{2}\log|P_I(\theta+n\omega,\omega,E)|+C_0+C|I|^{1-\sigma}.
  \end{gather*}
\end{proof}

\section{Elimination of Resonances}\label{sec:elimination}

As in the discrete case, the elimination of resonances is based on the
following result.
\begin{lemma}[{\cite[Lem. 9.9]{Bou05}}]\label{lem:Bourgain-slopes}
  Let $ \cS \subset [0,1]^{2n} $ be a semialgebraic set of degree
  $ B $ and $ \mes_{2n} \cS <\eta $,
  $ \log B \ll \log \frac{1}{\eta} $. We denote
  $ (\theta,\omega)\in [0,1]^n\times [0,1]^n $ the product variable.
  Fix $ \epsilon > \eta^{\frac{1}{2n}} $. Then there is a
  decomposition
  \begin{equation*}
    \cS= \cS_1 \cup \cS_2
  \end{equation*}
  $ \cS_1 $ satisfying
  \begin{equation*}
    \mes_n(\Proj_\omega \cS_1)< B^C\epsilon
  \end{equation*}
  and $ \cS_2 $ satisfying the transversality property
  \begin{equation*}
    \mes_n(\cS_2\cap L)<B^C\epsilon^{-1}\eta^{\frac{1}{2n}}
  \end{equation*}
  for any $ n $-dimensional hyperplane $ L $ s.t.
  $ \max_{0\le j\le n-1} |\Proj_L(e_j)|<\frac{1}{100}\epsilon $ (we
  denote by $ e_0,\ldots,e_{n-1} $ the $ \omega $-coordinate vectors).
\end{lemma}

Our elimination result is as follows.
\begin{proposition}\label{prop:elimination-fixed-phase}
  Assume that $ L(\omega,E)\ge \gamma>0 $ for all
  $ (\omega,E)\in \DC\times [E',E''] $. Let  $ I=[a,b] $, $ J=[a',b'] $.
  If
  \begin{equation*}
    |J|\simeq |I|^A\ge \max(|a|,|b|,|a'|,|b'|),\ |I|,A\ge C(V,d,\DC,E',E'',\gamma),
  \end{equation*}
  there exists a set
  \begin{equation*}
    \Theta_J,\ \mes \Theta_J\le \exp(-c|J|^{\alpha}),\ c=c(V,d,\DC,E',E''),\alpha=\alpha(d,\DC),
  \end{equation*}
  and for each $ \theta\in \T^d\setminus \Theta_J $ and $ B>0 $ there exists a set
  \begin{equation*}
    \Omega_{I,J,\theta,B},\ \mes \Omega_{I,J,\theta,B}\le |J|^{C-B},\ C=C(V,d,\DC,E',E''),
  \end{equation*}
  such that the following holds.
  For any $ \theta\in \T^d\setminus\Theta_J $, $ \omega\in \DC\setminus\Omega_{I,J,\theta,B} $, and $ E\in[E',E'']$
  we have that if
  \begin{equation*}
    \log\norm{M_J(\theta,\omega,E)}\le |J|L(\omega,E)-|J|^{1-\sigma/2},
  \end{equation*}
  then
  \begin{equation*}
    \log\norm{M_I(\theta+n\omega,\omega,E)}>|I|L(\omega,E)-|I|^{1-\sigma}
  \end{equation*}
  for any $|J|^B\le |n|\le\exp(c|I|^\sigma)$, $c=c(V,d,\DC,E',E'',A,\gamma)$ (recall that
  $ \sigma $ is as in \cref{thm:ldt}).
\end{proposition}
\begin{proof}
  Let $ \Theta_J $ be the set from \cref{prop:Cartan-E} with
  $ H=|J|^{\sigma/2}/C $. The measure estimate on $ \Theta_J $ holds
  with $ \alpha=\sigma/(4d+2) $.

  Fix $ \theta_0\in \T^d\setminus \Theta_J $. Consider the set $ \cB $
  of $ (\theta,\omega,E)\in \T^d\times \DC\times [E',E''] $ such that
  \begin{gather*}
    L(\omega,E)\ge \gamma,\\
	\log\norm{M_J(\theta_0,\omega,E)}\le |J|L(\omega,E)-|J|^{1-\sigma/2},\\
    \log\norm{M_I(\theta,\omega,E)}\le |I|L(\omega,E)-|I|^{1-\sigma}.
  \end{gather*}
  By \cref{prop:Lyapunov-convergence-rate} we have that
  $ \cB\subset \cB' $ where $ \cB' $ is defined by
  \begin{gather*}
    (\theta,\omega,E)\in \T^d\times \DC_{|J|^C}\times [E',E'']\\
    L_J(\omega,E)\ge \gamma/2,\\
	\log\norm{M_J(\theta_0,\omega,E)}\le |J|L_J(\omega,E)-|J|^{1-\sigma/2}/2,\\
    \log\norm{M_I(\theta,\omega,E)}\le
    |I|L_I(\omega,E)-|I|^{1-\sigma}/2.
  \end{gather*}
  By \cref{lem:semi-algebraic-all} we know that
  $ \cB' \subset \cS \subset \cB'' $ with $ \cS $ semialgebraic of
  degree less than $ |J|^C $ and $ \cB'' $ defined by
  \begin{gather*}
    (\theta,\omega,E)\in \T^d\times \DC_{|J|^C}\times [E',E'']\\
    L_J(\omega,E)\ge \gamma/4,\\
	\log\norm{M_J(\theta_0,\omega,E)}\le |J|L_J(\omega,E)-|J|^{1-\sigma/2}/4,\\
    \log\norm{M_I(\theta,\omega,E)}\le
    |I|L_I(\omega,E)-|I|^{1-\sigma}/4.
  \end{gather*}

  To get the conclusion we want
  $ (\{ \theta_0+n\omega \},\omega,E)\notin \cB $ for $ \omega $
  outside an exceptional set and all $ E\in[E',E''] $. It is enough to
  argue that
  $ (\{ \theta_0+n\omega \},\omega)\notin
  \cS':=\Proj_{(\theta,\omega)} \cS $
  for $ \omega $ outside an exceptional set.  We achieve this by
  invoking \cref{lem:Bourgain-slopes}. By the Tarski-Seidenberg
  principle (see \cite[Prop. 9.2]{Bou05}) the set $ \cS' $ is known to
  be semialgebraic of degree less than $ |J|^C $. We need to estimate
  $ \mes (\cS') $.

  We have
  \begin{equation*}
    \mes (\cS') \le \mes (\Proj_{(x,\omega)} \cB'').
  \end{equation*}
  Let $ \Omega_{J,\theta_0} $,
  $ \mes(\Omega_{J,\theta_0})\le \exp(-c|J|^\alpha) $ be the set from
  \cref{prop:Cartan-E}.  Consider the set
  \begin{equation*}
    \Theta''=\Proj_\theta \{ (\theta,\omega,E)\in \cB'': \omega\notin \Omega_{J,\theta_0}  \}.
  \end{equation*}
  If $ (\theta,\omega,E)\in \cB'' $ and
  $ \omega\notin \Omega_{J,\theta_0} $ then by \cref{prop:Cartan-E} we
  have that $ E\in \cE_{J,\theta_0,\omega} $ which is the union of
  less than $ \exp(C(\log|J|)^{2/\sigma}) $ intervals each having
  measure less than $ \exp(-c|J|^{\alpha}) $. If $ A $ is large enough
  so that $ |J|^{\alpha}\ge |I|^2 $ then \cref{thm:ldt} and
  \cref{lem:stability-rough} imply that
  \begin{equation*}
    \mes(\Theta'')\le \exp(C(\log|J|)^{2/\sigma}) \exp(-c|I|^{\sigma})\le \exp(-c(A)|I|^{\sigma}).
  \end{equation*}
  We conclude that
  \begin{equation*}
    \mes ( \cS') \le \mes (\Proj_{(x,\omega)} \cB'')
    \le \mes(\Omega_{J,\theta_0})+\mes(\Theta'')\le \exp(-c|I|^{\sigma}).
  \end{equation*}

  Let
  \begin{equation*}
    \cS'=\cS_1'\cup \cS_2'
  \end{equation*}
  be the decomposition afforded by \cref{lem:Bourgain-slopes} with
  $ \epsilon=200/|J|^B $.  The set of $ \{ \theta_0+n\omega \} $ with
  $ \omega\in [0,1]^d $ is contained in a union of hyperplanes
  $ L_{n,\alpha} $, $ \alpha\le |n|^d $. The hyperplanes
  $ L_{n,\alpha} $ are parallel to the hyperplane
  $ (n\omega,\omega) $, $ \omega\in\R^d $, and therefore
  \begin{equation*}
    |\Proj_{L_{n,\alpha}} e_j|\le \frac{1}{|n|}<\frac{\epsilon}{100}\text{ for all } \alpha,e_j, |n|\ge |J|^B
  \end{equation*}
  ($ e_j $ are as in \cref{lem:Bourgain-slopes}).  The conclusion
  follows by letting
  \begin{equation*}
    \Omega_{I,J,\theta_0,B}= \{ \omega : (\omega,\{ \theta_0+n\omega  \})\in \cS' \text{ for some }
    |J|^B\le |n|\le\exp(c|I|^{\sigma})  \}.
  \end{equation*}
  Note that by \cref{lem:Bourgain-slopes} we have
  \begin{multline*}
    \mes(\Omega_{I,J,\theta_0,B})\le \mes(\Proj_\omega \cS_1')+\sum_{\alpha,n} \mes(\cS_2'\cap L_{n,\alpha})\\
    \lesssim \frac{|J|^C}{|J|^B}+\sum_n |n|^d |J|^{C+B}\exp(-c|I|^{\sigma})\le \frac{|J|^C}{|J|^B}
  \end{multline*}
  provided the constant $ c $ in the upper bound of $ |n| $ is small
  enough.
\end{proof}

\begin{remark}\label{rem:discrete-vs-continuous-elimination}
  We assume the notation from the proof of the previous Proposition.
  Following the discrete case strategy (see \cite[Ch. 10]{Bou05}) we
  could set things up so that the set $ \cB $ used for elimination is
  determined by
  \begin{equation}\label{eq:discrete-elimination}
    \begin{gathered}
      v_{a'}(b';\theta_0,\omega,E)=0,\\
      \log\norm{M_I(\theta,\omega,E)}\le
      |I|L(\omega,E)-|I|^{1-\sigma},
    \end{gathered}
  \end{equation}
  where $ v_{a'}(b';\theta_0,\omega,E) $ plays the same role as the finite
  volume Dirichlet determinant did in the discrete case.  The benefit
  of this set-up is that the first equation restricts $ E $ to a
  finite set of values (the eigenvalues in the interval $ [E',E''] $),
  which lets us project onto $ (\theta,\omega) $ and get a small set (of
  course, we would also need an estimate for the number of
  eigenvalues in $ [E',E''] $). The problem is that when we pass to
  the semialgebraic approximation the first equality becomes an inequality and the
  previous reasoning breaks. In the discrete case one approximates the
  potential by a polynomial $ \tilde V $ and as a result one gets a
  new operator $ \tilde H $ to which one can apply the reasoning that
  leads to \cref{eq:discrete-elimination} for the semialgebraic
  approximation in the same way as for $ H $. This doesn't work in the
  continuous case because we are forced to approximate the solutions,
  rather than the potential.
\end{remark}

\section{Proof of the Main Result}\label{sec:main-proof}

We are now ready to prove \cref{thm:localization}.

Let $ N_0=N_0(V,d,E',E'',\gamma) $ and $ C_0=C_0(V,d,E',E'',\gamma) $
be large enough and define $ N_k=(N_{k-1})^{C_0} $, $ k\ge 1 $. Let
$ \Theta_k $, $ \Omega_{k,\theta} $ be the sets from
\cref{prop:elimination-fixed-phase} with $ J_k=[-N_{k+1},N_{k+1}] $,
$ I_k=[-N_k,N_k] $ and $ B $ such that
$ |J_k|^B\in[N_{k+2}/4, N_{k+2}/2 ] $, Note that we have
\begin{equation*}
  \mes (\Theta_k)\le \exp(-cN_{k+1}),\ \mes(\Omega_{k,\theta})\le N_{k+2}^{-1/2}.
\end{equation*}
Let
\begin{equation*}
  \Theta=\bigcap_{\underline k=0}^{\infty} \bigcup_{k\ge \underline k} \Theta_k.
\end{equation*}
Given $ \theta\in \T^d\setminus\Theta $ there exists $ k_0 $ such that
$ \theta\in \T^d\setminus \Theta_k $, $ k\ge k_0 $.  Let
\begin{equation*}
  \Omega_\theta=\bigcap_{\underline k=k_0}^{\infty} \bigcup_{k\ge \underline k} \Omega_{k,\theta}.
\end{equation*}
By Borel-Cantelli we clearly have that
$ \mes(\Theta)=\mes(\Omega_\theta)=0 $.

Let $ \theta\in\T^d\setminus \Theta $,
$ \omega\in\DC\setminus \Omega_{\theta}$. It is well known that the
energies with polynomially bounded solutions are dense in the spectrum
(see \cite[Cor. C.5.5]{Sim82}). So, given $ E\in[E',E''] $ so that
there exists $ y\not \equiv 0 $ satisfying $ H(\theta,\omega)y=Ey $
and
\begin{equation}\label{eq:polynomial-bound}
  |y(t)|\le(1+|t|)^C,
\end{equation}
it is enough to show that $ y $ decays exponentially.

If
\begin{equation*}
  \log\norm{M_{J_k}(\theta,\omega,E)}> |J_k|L(\omega,E)-|J_k|^{1-\sigma/2}
\end{equation*}
for infinitely many $ k $, then \cref{prop:Green-decay} together with
Poisson's formula and \cref{eq:polynomial-bound} imply that
$ y\equiv 0 $. Therefore, for $ k $ large enough we must have
\begin{equation*}
  \log\norm{M_{J_k}(\theta,\omega,E)}\le |J_k|L(\omega,E)-|J_k|^{1-\sigma/2}
\end{equation*}
and by \cref{prop:elimination-fixed-phase}
\begin{equation*}
  \log\norm{M_{I_k}(\theta+n\omega,\omega,E)}> |I_k|L(\omega,E)-|I_k|^{1-\sigma},\ N_{k+1}/2\le |n|\le 2N_{k+2}. 
\end{equation*}
Using \cref{prop:Green-decay} we can iterate Poisson's formula as in
\cref{lem:Poisson-iteration} and get that
\begin{equation*}
  |y(t)|\le (1+|t|)^C\exp(-c|t|L(\omega,E))\le \exp(-c|t|L(\omega,E)/2),\ |t|\in[N_{k+1},N_{k+2}]. 
\end{equation*}
This concludes the proof.

\appendix \label{sec:appendix}

\section{Appendix} \label{sec:appendix} 

Before we prove the large deviations estimate we need to recall the following result from \cite{GolSch01}.
\begin{theorem}[{\cite[Thm. 8.5]{GolSch01}}]\label{thm:GS-Cartan-estimate}
  Let $ d $ be a positive number.
  Suppose $ u:D(0,2)^d\to[-1,1] $ is subharmonic in each variable. Given $ r\in(0,1) $
  there exists a polydisk
  \begin{equation*}
    \Pi=D(x_1^{(0)},r)\times \ldots \times D(x_d^{(0)},r) \subset \C^d
  \end{equation*}
  with $ x_1^{(0)},\ldots,x_d^{(0)}\in[-1,1] $ and a Cartan set $ \cB\in \Car_d(H) $, $ H=\exp(-r^{-\beta}) $
  so that
  \begin{equation}\label{eq:GS-deviation}
    |u(z)-u(z')|\lesssim r^\beta\text{ for all } z,z'\in\Pi\setminus\cB.
  \end{equation}
  The constant $ \beta>0 $ depends only on the dimension $ d $.
\end{theorem}

We will also need the following fact about the discrepancy of the sequence of shifts of a Diophantine vector.
Let $ R=\prod_i [a_i,b_i]\subset [0,1]^d $. It is known (see \cite{Hla73}) that for $ \omega\in \DC_N $ we have
\begin{equation}\label{eq:discrepancy}
  \# \{ n : n\omega \in R, 1\le n \le N  \}= N \Vol(R)+C(d,DC) O(N^{1-1/A}\log^2 N).
\end{equation}

\begin{proof}[Proof of \cref{thm:ldt}]
  Let
  \begin{equation*}
    u(\theta)=\frac{\log\norm{M_{I}(\theta_0+\rho \theta+i\eta,\omega,E)}}{C |I|},\quad
    v(\theta)=\frac{\log\norm{M_{I}(\theta+i\eta,\omega,E)}}{C |I|},
  \end{equation*}
  where $ \theta_0=(1/2,\ldots,1/2)\in \T^d $.
  We choose $ \rho=\rho(V) $ such that $ u $ is defined on $ D(0,2)^d $ and $ C=C(V,|E|) $ such that
  \begin{equation*}
    |u(\theta)|\le 1, \theta\in D(0,2)^d
  \end{equation*}
  and
  \begin{equation}\label{eq:v-invariance}
    |v(\theta)|\le 1, |v(\theta)-v(\theta+\omega)|\le \frac{1}{|I|}, \theta\in \T^d
  \end{equation}
  (recall that we have \cref{eq:transfer-matrix-bound} and \cref{eq:almost-invariance}).
  Applying \cref{thm:GS-Cartan-estimate} to $ u $ with $ r\in(0,1) $ we get that there exists
  $ R=x^{(0)}+[-\rho r,\rho r]^d\subset [0,1]^d $ such that
  \begin{equation}\label{eq:v-deviation}
    |v(\theta)-v(\theta')|\lesssim r^\beta, \theta,\theta'\in R\setminus \cB',
  \end{equation}
  with $ \mes(\cB')\lesssim d \rho^d \exp(-r^{-\beta}) $. Note that in terms of the notation of
  \cref{thm:GS-Cartan-estimate} we have $ R=\Pi\cap \R^d $, $ \cB'=\cB\cap \R^d $, and the measure
  estimate for $\cB' $ follows from \cref{lem:Cartan-measure}.

  It follows from \cref{eq:discrepancy} that for any $ \theta\in \T^d $ there exists
  \begin{equation*}
    k\le k_0:=[C(V,d,\DC) r^{-2dA}] \text{ such that }\theta+k\omega\in R
  \end{equation*}
  (the factor of $ 2 $ in the exponent of $ r $ can be replaced by $ 1+\epsilon $).
  Therefore, as a consequence of \cref{eq:v-invariance} and \cref{eq:v-deviation} we have
  \begin{equation*}
    |v(\theta)-v(\theta')|<C(V,d,\DC) \left( r^\beta + \frac{r^{-2dA}}{|I|} \right),
    \theta,\theta'\in \T^d\setminus \tilde \cB
  \end{equation*}
  with
  \begin{equation*}
    \tilde \cB := \cup_{k=0}^{k_0} (\cB'+k\omega), \quad \mes(\tilde \cB)\le C(V,d,\DC) r^{-2dA}\exp(-r^{-\beta}).
  \end{equation*}
  Taking
  \begin{equation}\label{eq:r-bounds}
    |I|^{-\frac{1}{2dA+\beta}}\le r \le c(V,d,\DC)
  \end{equation}
  we have
  \begin{equation*}
    |v(\theta)-v(\theta')|<C r^\beta, \theta,\theta'\in\T^d\setminus\tilde \cB, \quad
    \mes(\tilde \cB)\le \exp(-r^{-\beta}/2).
  \end{equation*}
  It is now straightforward to see that
  \begin{multline*}
    \mes \{\theta : |\log\norm{M_{I}(\theta+i\eta,\omega,E)}-|I| L_{I}(\eta,\omega,E)|>C |I| r^{\beta}  \}\\
    \le \exp(-r^{-\beta}/2), C=C(V,d,\DC,|E|).
  \end{multline*}
  The conclusion follows immediately by choosing $ r $ so that
  \begin{equation*}
    C |I| r^{\beta} = \epsilon |I|^{1-\sigma}.
  \end{equation*}
  Note that due to \cref{eq:r-bounds} we need to take $ \sigma<\beta/(2dA+\beta) $.  
\end{proof}

We use the following simple result to pass from continuous time Lyapunov exponents to discrete time Lyapunov exponents.
\begin{lemma}\label{lem:Lyapunov-I-to-N}
  Let $ I=[a,b] $. Then for any $ \omega\in\T^d $, $ E\in \C $, $ \eta\in\R^d $, $ \norm{\eta}\le \rho(V) $,
  $ n\in \Z  $, $ n\ge 1 $ we have
  \begin{equation*}
    |L_I(\eta,\omega,E)-L_n(\eta,\omega,E)|\le \frac{C(V,|E|)}{|I|}(|n-|I||+2).
  \end{equation*}
\end{lemma}
\begin{proof}
  By \cref{eq:M-semigroup} and the bounds on the transfer matrix and its inverse we have
  \begin{equation*}
    |\log\norm{M_I}-\log\norm{M_J}|\le C \left|I\cup J\setminus (I\cap J)\right|, C=C(V,|E|)
  \end{equation*}
  for any other closed finite interval $ J $. The conclusion follows by applying this fact with
  $ J=[[a],[a]+n] $ and the definition of the finite scale Lyapunov exponents.
\end{proof}

\begin{proof}[Proof of \cref{prop:Lyapunov-convergence-rate}]
  With the same proof as that of \cite[Lem. 10.1]{GolSch01} we have
  \begin{equation}\label{eq:L_n-vs-L}
    0\le L_n(\eta,\omega,E)-L(\eta,\omega,E)\le C(V,d,\DC,|E|) \frac{(\log n)^{1/\sigma}}{n}.
  \end{equation}
  Note that in fact the proof of \cite[Lem. 10.1]{GolSch01} only points out the adjustments that need
  to be made to the proof of \cite[Lem. 4.2]{GolSch01}; up to these adjustments the proof of
  \cite[Lem. 4.2]{GolSch01} works as is for our setting too. Furthermore, from the proof of
  \cite[Lem. 4.2]{GolSch01} we also have that
  \begin{equation}\label{eq:L_n-vs-L_2n}
    |L_n(\eta,\omega,E)-L_{2n}(\eta,\omega,E)|\le C(V,d,\DC,|E|)\frac{(\log n)^{1/\sigma}}{n}.
  \end{equation}
  The proof of \cref{eq:L_n-vs-L_2n} only relies on the large deviations estimate at scale
  $ \ell\simeq (\log n)^{1/\sigma} $ and therefore it is enough to have $ \omega\in \DC_\ell\supset \DC_n $.
  Furthermore, one only needs that $ L_{\ell}(\eta,\omega,E)\gtrsim \gamma $ and to have this it is enough to
  assume $ L_{n}(\eta,\omega,E)\ge \gamma $ (due to \cref{lem:Lyapunov-increasing}).
  
  The conclusions follow from \cref{eq:L_n-vs-L_2n} and \cref{eq:L_n-vs-L} together with \cref{lem:Lyapunov-I-to-N}.
\end{proof}

\begin{proof}[Proof of \cref{lem:Lyapunov-increasing}]
  Let $ m=[|I|]+2 $, $ n=[|J|]+1 $.
  We have $ m=kn+r $ and by subadditivity
  \begin{equation*}
    mL_m\le knL_n+rL_r.
  \end{equation*}
  It follows that
  \begin{multline*}
    L_n\ge L_m+\frac{r(L_m-L_r)}{kn}\ge L_m-\frac{r C(V,|E|)}{kn}
    \ge L_m-\frac{nC(V,|E|)}{m-n}\\
    \ge L_m-C(V,|E|)\frac{|J|+1}{|I|-|J|}.
  \end{multline*}
  The conclusion follows from \cref{lem:Lyapunov-I-to-N}.
\end{proof}

\bibliographystyle{alpha}
\bibliography{../Schroedinger}

\medskip
\noindent \textbf{Ilia Binder}: 
Department of Mathematics, University of Toronto,
Toronto, ON, M5S 2E4, Canada, \href{mailto:ilia@math.utoronto.ca}{ilia@math.utoronto.ca}

\medskip
\noindent \textbf{Damir Kinzebulatov}:
Department of Mathematics and Statistics, McGill University,  
Montreal, QC, H3A 0B9,
Canada,  \href{mailto:damir.kinzebulatov@mail.mcgill.ca}{damir.kinzebulatov@mail.mcgill.ca}

\medskip
\noindent \textbf{Mircea Voda}: Department of Mathematics, The University of Chicago,
Chicago, IL, 60637, U.S.A.,
\href{mailto:mircea.voda@math.uchicago.edu}{mircea.voda@math.uchicago.edu}

\end{document}